\documentclass[12pt]{article}

\usepackage{amssymb,amsmath,amsthm,mathrsfs,cite,bm}
\usepackage{color} 
\usepackage{verbatim}
\usepackage{soul}
\usepackage{multirow}
\usepackage{caption}
\usepackage{graphics}
\usepackage{threeparttable}
\usepackage{pifont}
\usepackage{float}
\usepackage{cases}
\RequirePackage[colorlinks,citecolor=blue]{hyperref}

\newtheorem{theorem}{Theorem}
\newtheorem{lemma}{Lemma}

\newtheorem{proposition}{Proposition}
\newtheorem{example}{Example}

\definecolor{gray}{gray}{0.55}
\newcommand\rev{\color{blue}}

\begin{document}

	\begin{center}
		{\Large {\bf A Non-convex Optimization Approach of Searching Algebraic Degree Phase-type Representations for General Phase-type Distributions}}\\[0.1in]
		Yujie Liu\footnote{Department of Industrial Systems Engineering \& Management, National University of Singapore, 117576, Singapore; yj-liu@nus.edu.sg} $\bullet$
		Dacheng Yao\footnote{Academy of Mathematics and Systems Science, Chinese Academy of Sciences, Beijing, 100190, China; School of Mathematical Sciences, University of Chinese Academy of Sciences, Beijing, 100049, China; dachengyao@amss.ac.cn} $\bullet$
		Hanqin Zhang\footnote{Department of Analytics \& Operations, National University of Singapore, 119245, Singapore; bizzhq@nus.edu.sg}
	\end{center}
	
	\noindent
	{\bf Abstract:}
	For a continuous-time phase-type distribution, starting with its Laplace-Stieltjes transform,
	we obtain a necessary and sufficient condition for its minimal phase-type representation to have the same order as the algebraic degree of
	the Laplace-Stieltjes transform. To facilitate finding this minimal representation, we transform this condition equivalently into a quadratic nonconvex optimization problem, which can be effectively addressed using an alternating minimization algorithm. The algorithm convergence is also proved. Moreover, the method we develop for the continuous-time phase-type distributions can be directly used to the discrete-time phase-type distributions
	after establishing an equivalence between the minimal representation problems for continuous-time and discrete-times phase-type distributions.
	
	\
	
	
	\
	

	\section{Introduction}
	
	A phase-type (PH) distribution is the probability distribution of the absorption time of a finite-state Markov chain
	with all the states being transient except only one absorbing state. A PH distribution is called a continuous-time (discrete-time) PH distribution  if its corresponding Markov chain is continuous-time (discrete-time).  PH distributions have two advantages: partially keeping the memoryless property uniquely possessed by the exponential distribution, and accurately approximating any probability distributions given by nonnegative random variables under the topology of the weak convergence, see  Latouche and Ramaswami
	\cite{latouche1999}, and Neuts \cite{neuts1981matrix}. Because of these two advantages, PH distributions have been widely used in stochastic modeling, and make the system performance measures to become analytically and numerically tractable. For instance, they are used in queueing theory to model the customer interarrival/service times, see Bladt \cite{bladt1996}, and  Dai et al. \cite{dai2010}; they are implemented in reliability, risk management and computer communication to characterize the component lifetime (Latouche and Taylor \cite{latouche2002}), ruin probability (Asmussen and Albrecher \cite{asmussen2010ruin}), and data package transmission time (Alfa and Li \cite{alfa2002}), respectively; and they are also formulated to be product replacement times in stochastic inventory management, see He \cite{he2014fundamentals}, and Song and Zipkin \cite{song1993}.
	
	It is well-known that the absorbing times determined by different Markov chains may have the same distribution, this implies that for a given PH distribution, there exist multiple Markov chains such that their absorbing time distributions are the same as the given PH distribution,
	see Neuts \cite{neuts1981matrix}, and O'Cinneide \cite{o1989non}. Each such Markov chain is called a {\it PH representation} of this given PH distribution, and its states are called phases. Usually, the effectiveness and efficiency of the usage of the PH distributions in stochastic models depend on the number of the phases of their representations. A smaller number of phases lead to shorter computational time and higher accuracy in evaluating system performance measures. Motivated by this reason, a natural and important question we would ask is that for a PH distribution, how to determine its one representation with the smallest number of the phases among all its equivalent PH representations, that is, {\it the minimal PH representation problem}. The minimum number of the phases needed to represent it is called its {\it order}.
	
	In the past few decades,  a number of studies on the minimal PH representation problem have been carried out. Using the invariant polytope approach introduced by Dmitriev and Dynkin \cite{dmitriev1945,dmitriev1946}, for a continuous-time PH distribution, O'Cinneide \cite{o1991phase} gives a lower bound on its order by considering its Laplace-Stieltjes transform (LST). We know that
	LST of each PH distribution is a rational function. Starting with its LST for a PH distribution, Commault and Chemla \cite{commault1996invariant} prove that when its LST can be expressed as a ratio of two coprime polynomials, if the numerator polynomial turns out to be constant or linear (the degree of the numerator polynomial is zero or one), and all the roots of the denominator polynomial (the poles of its LTS) are real, then its order is equal to the degree of the denominator polynomial. In the theory of PH distributions, for a PH distribution, when its LST is written as the ratio of two coprime polynomials, the degree of the denominator polynomial is called its {\it  algebraic degree}, and its algebraic degree is always smaller than or equal to the order of its PH representation,  see Fackrell et al. \cite{fackrell2010}, and O'Cinneide \cite{o1991phase}.
	Further, they demonstrate that even when the degree of the numerator polynomial is two while keeping all the roots of the denominator polynomial to be real, or the roots of the denominator polynomial are not real while keeping the numerator polynomial to be constant/linear, the minimal PH representation problem becomes much more complicated, its order can get dramatically larger compared with its algebraic degree.
	
	With considering the difficulty of finding the minimal PH representation, people have to lower down this goal and switch to exploring some properties on PH distributions, which can help us to simplify some analyses on the Markov models formulated by PH distributions.
	Aldous and Shepp \cite{david1987least} establish the lower bound of the coefficient of the variation of PH distributions by studying the quadratic variation of the martingale formed by the absorbing time of the corresponding Markov chain. Using the majorization order introduced by  Marshall and  Olkin \cite{marshall1979inequalities},  O'Cinneide \cite{omajor1991} proves that when fixing mean, the order-$n$ PH distribution always majorizes the order-$n$ Erlang distribution. At the same time, people have been trying to find other PH representations with simpler structure such as triangular, Coxian, and unicyclic types for PH distributions. O'Cinneide \cite{o1991phase} obtains that any PH distribution whose LST
	has only real poles has a triangular PH representation. He and Zhang \cite{he2005note,he2006spectral,he2006ph,he2008algorithm}
	introduce spectral polynomial algorithms for computing ordered Coxian
	representations for PH-distributions with only real poles.  Telek and Horv\'ath  \cite{telek2007} propose the moments matching method to compute the minimal PH representation. M\'esz\'aros et al. \cite{telek2013}, and  Telek and Horv\'ath  \cite{telek2007}	propose the moments matching ({\bf MM}) method to compute the minimal PH representation.
	
	It  is worth mentioning that the minimal representation problem can be viewed as a minimal positive realization problem in the sense that LST is
	considered as a transfer function in a linear system, see Commault \cite{commault2003phase}. The minimal positive realization problem has been received significant attention in the control society, but most of them focus on special transfer functions, see  Benvenuti \cite{Benvenuti2022}, Liu et al. \cite{LiuYaoZhang2023}, and  reference therein.
	
	In sum, the current state of affairs is far from solving the minimal PH representation problem. This is also evidenced from the survey papers by
	Benvenuti \cite{Benvenuti2022},  Commault and Mocanu \cite{commault2003phase}, and O'Cinneide \cite{o1993triangular}.
	To push the current status toward this goal,
	in this paper, starting with the LST, we establish an if and only if (iff) condition for any given PH distribution to have a minimal representation with the same order as its algebraic degree. The iff condition is characterized by the zero-optimal value for a non-convex optimization problem. To get the iff condition, we first show the uniqueness for the Jordan form generated by the poles from the LST, and develop a unique expression of the LST in terms of this Jordan form.
	The unique expression enables us to rephrase the iff condition into the non-convex optimization problem. To effectively solve this optimization problem, we develop an alternating minimization ({\bf AM}) algorithm to compute the minimal PH representation by iteratively solving two quadratic optimization problems, and prove the convergence of the algorithm. In order to establish the algorithm convergence, we use the solution representation theorem for the quadratic optimization problem to create a single-variable quadratic function characterized by the inner product of the optimal solution. The solution boundedness is given through exploring this quadratic function, and the algorithm convergence is then obtained by this boundedness.  Further, we establish the equivalence between the minimal representation problems for continuous-time  and discrete-time PH distributions. This equivalence allows us to transform the problem of finding a minimal representation for a discrete-time PH distribution into the continuous-time case. Thus the method developed in this paper provides a unified approach for both continuous-time and discrete-time minimal PH representation problems.

	
	The rest of this paper is organized as follows. In Section \ref{sec:problem}, we formulate our minimal representation problems based on LSTs. Section \ref{cts_case} focuses on the continuous-time minimal PH representation problem.
	In Section \ref{dct_case}, we establish the equivalence between the minimal representation problems for continuous-time and discrete-time PH distributions, and the equivalence enables us to directly adopt the method established in Section \ref{cts_case} to obtain the minimal PH representation for the discrete-time PH distributions.    
	In Section \ref{sec:numerical}, we present a numerical study using the developed algorithm. Finally, we conclude the paper in Section \ref{sec:conclude}.
	
	Throughout the rest of the paper, ${\cal C}$ represents the complex field and $\imath$ is the imaginary unit.
	For $z\in {\cal C}$, its real part, imaginary part and modulus are denoted by ${\sf Re}(z)$, ${\sf Im}(z)$, and $|z|$, respectively. ${\cal R}^n$ denotes $n$-dimensional real space and ${\cal R}_+^n$ denotes the set of $n$-dimensional nonnegative vectors. For $x,y\in {\cal R}^n$, $\langle x,y\rangle$ denote their inner product. ${\cal R}^{m\times n}$ is the set of all $m\times n$ real matrices, and ${\cal R}_+^{m\times n}$ is the set of all nonnegative $m\times n$ real matrices. The transpose of a vector or a matrix is denoted by appending a superscript ``$\top$".
	All the vectors in this paper are understood to be row vectors.
	For $A=(a_{ij})_{n\times n}\in {\cal R}^{n\times n}$, its Frobenius norm is represented by $\|A\|_F$. ${\sf adj}(A)$ denotes its adjugate matrix, and ${\sf det}(A)$ denotes its determinant.
	Let ${\bf 1}$ be the vector with each component being one, ${\bf 0}$ the vector or matrix whose whole components are zero, and $I$ the identity matrix. Their dimensions would be clear from the context in which they are located.
	For $A_i\in {\cal R}^{n_i\times n_i}$ with $1\leq i\leq k$,
	\begin{eqnarray*}
		{\sf diag}\Big(A_1,\cdots,A_k\Big) =\begin{bmatrix}
			A_1& & \\
			& \ddots& \\
			& & A_k
		\end{bmatrix}.
	\end{eqnarray*}

	\section{ Minimal PH Representation Problem: Continuous-time Case}
	\label{sec:problem}
	
	For the continuous-time case, in view of the work by Asmussen and Bladt \cite{asmussen-bladt-1997}, and O'Cinneide \cite{o1990characterization},
	for any probability distribution $F(\cdot)$ on $[0, \infty)$ with continuous positive density on $(0, \infty)$, its LST is given by the rational function
	\begin{align}\label{prob1}
		{\cal L}(s)=\int^\infty_{0-}e^{-st}{\sf d} F(t)= \frac{p(s)}{q(s)}=\frac{p_{n-1}s^{n-1}+p_{n-2}s^{n-2}+\cdots+p_0}{s^n+q_{n-1}s^{n-1}+\cdots+q_0}, \
		{\sf Re}(s)\geq 0.
	\end{align}
	If ${\cal L}(s)$ satisfies that
	\begin{itemize}
		\item [({\bf A1})] its coefficients are real numbers;
		\item[({\bf A2})] $p(s)$ and $q(s)$ are coprime;
		\item [({\bf A3})] ${\cal L}(0)=1$ and ${\cal L}(s)$ has a unique pole of maximal real part,
	\end{itemize}
	then $F(\cdot)$ is a PH distribution characterized by the absorbing time of the continuous-time Markov chain on the states ${\cal S}=\{1,\ldots, m, m+1\}$ with $m\geq n$.
	Specifically, the Markov chain's first $m$ states are all transient while the last state is absorbing, the initial state distribution is $\mathbf{\alpha}=(\alpha_1,\ldots,\alpha_m)$ with $\alpha{\bf 1}^\top=1$, and the infinitesimal generator is
	\[  \begin{bmatrix}
		A&-A {\bf 1}^\top\\
		{\bf 0}& 0
	\end{bmatrix},
	\]
	where $m\times m$ matrix $A=(a_{ij})_{m\times m}$ satisfies $a_{ii}<0$ for  $1\leq i\leq m$, and $a_{ij}\geq 0$ for $i\neq j$.
	In particular, $F(t)=1-\alpha e^{At} {\bf 1}^\top$ for $t\geq 0$.
	Then ${\cal L}(\cdot)$ can be written as
	\begin{align}\label{equ:lapla}
		{\cal L}(s)=\int_{0}^{\infty}-e^{-st}\alpha Ae^{At}{\bf 1}^\top \ {\sf d} t=-\alpha A(sI-A)^{-1}{\bf 1}^\top=\frac{-\alpha A\cdot {\sf adj}(sI-A){\bf 1}^\top}{{\sf det}(sI-A)}
	\end{align}
	for ${\sf Re}(s)\geq 0$.
	We call such $(\alpha, A)$ an {\it $m$-order continuous-time PH representation} of $F(\cdot)$ with LST ${\cal L}(s)$, and
	$n$ {\it the algebraic degree} of the PH distribution $F(\cdot)$. In the following, for the continuous-time case,
	we start with LST ${\cal L}(s)$ given by (\ref{prob1}) satisfying ({\bf A1})-({\bf A3}) to identify {\it when
		there exists such a Markov chain with $(n+1)$ states, that is, when there exists a PH representation with the same order as its algebraic degree. Further, if it exists, we develop an algorithm to determine the initial distribution $\alpha=(\alpha_1,\ldots,\alpha_n)$ and the infinitesimal subgenerator
		$A=(a_{ij})_{n\times n}$.}
	
	Before developing our analysis on the minimal representation, we comment on ${\cal L}(s)$ given by (\ref{prob1}). ${\cal L}(s)$ in (\ref{prob1}) assumes that there is no probability mass at zero. That is, the Markov chain we identify based on ${\cal L}(s)$ initially starts at the transient states with probability one. This assumption can be dropped by directly adding a constant term (the probability of initially starting with the absorbing state) to ${\cal L}(s)$.
	Our following analysis and results on ${\cal L}(s)$ given by (\ref{prob1})  can be directly adopted to the case with positive probability mass at zero by just setting the initial probability at the absorbing state to be this positive probability mass at zero
	for the Markov chain we have identified. Without losing generality, we only focus on ${\cal L}(s)$ given by (\ref{prob1}) in the rest of the paper.
	
	\section{Alternating Minimization Algorithm}\label{cts_case}	
	Consider the rational function ${\cal L}(s)$ given by (\ref{prob1}) with ({\bf A1})-({\bf A3}) holding.
	Suppose that all distinct roots of $q(s)$ in ${\cal L}(s)$ are
	\begin{align}\label{pole}
		\left\{\lambda_1, \lambda_2,\ldots,\lambda_{\ell}, \mu_{\ell+1}\pm \imath\omega_{\ell+1}, \mu_{\ell+2}\pm \imath\omega_{\ell+2},\ldots, \mu_{k}\pm \imath\omega_{k}  \right\}
	\end{align}
	with $\lambda_i,\mu_j,\omega_j\in {\cal R}$, $\omega_j\neq 0$ for $1\leq i\leq \ell, \ell+1\leq j  \leq k$, and their multiplicities are $n_1,n_2,\ldots,n_k$, respectively. That is,
	\begin{align}\label{equ_q}
		\begin{aligned}
			q(s)&=\prod_{i=1}^{\ell}(s-\lambda_i)^{n_i}\cdot\prod_{j=\ell+1}^{k}\Big((s-\mu_j-\imath \omega_j)(s-\mu_j+\imath \omega_j)\Big)^{n_j}\\
			&=\prod_{i=1}^{\ell}(s-\lambda_i)^{n_i}\cdot\prod_{j=\ell+1}^{k}\Big((s-\mu_j)^2+\omega_j^2\Big)^{n_j}.
		\end{aligned}
	\end{align}
	Let $\lambda_1=\arg\max_{1\leq i\leq \ell}\{\lambda_i\}$. Frorm ({\bf A3}), we have
	\begin{eqnarray}
		\lambda_1>\max_{2\leq i\leq \ell \atop \ell+1\leq j  \leq k}\{\lambda_i,\mu_j\}.\label{A3}
	\end{eqnarray}
	In view of \eqref{equ:lapla},  our problem can be equivalently rephrased as follows:
	\begin{itemize}
		\item[] {\it Finding  $\alpha \in {\cal R}_+^n$ with $\alpha{\bf 1}^\top=1$, and
			nonsingular matrix $A=(a_{ij})_{n\times n}\in {\cal R}^{n\times n}$ with $a_{ii}<0$ for $1\leq i\leq n$, $a_{ij}\geq 0$ for $i\neq j$,
			and $\sum_{j=1}^na_{ij}\leq 0$ for $1\leq i\leq n$
			such that
			\begin{align}
				&\label{equ_Lapl} {\cal L}(s)=\frac{p(s)}{q(s)}=-\alpha A(sI-A)^{-1}{\bf 1}^\top=\frac{-\alpha A\cdot {\sf adj}(sI-A){\bf 1}^\top}{{\sf det}(sI-A)}.
			\end{align}
		}
	\end{itemize}
	
	Recall (\ref{equ_q})
	and note that ${\sf det}(sI-A)$ is a polynomial of $n$ degree. To ensure \eqref{equ_Lapl} holds, a necessary condition is that $q(s)={\sf det}(sI-A)$. That is, the roots (including multiplicity) of $q(s)$ must be the spectrum of $A$. Let $J$ be the Jordan form of $A$. Then there exists a nonsingular matrix $P$ such that
	\begin{align}\label{equ_Jordanform-1}
		A=P^{-1}JP.
	\end{align}
	Thus
	\begin{align}\label{equ_Jordanform}
		-\alpha A(sI-A)^{-1}{\bf 1}^\top=-\alpha P^{-1} J(sI-J)^{-1}P{\bf 1}^\top.
	\end{align}
	
	In the following, based on (\ref{equ_Jordanform-1})-(\ref{equ_Jordanform}), we first explore some properties on the Jordan form $J$ and matrix $P$.
	Using these properties, we transform our problem into the problem on how to identify whether there exists a feasible solution to the set of linear and quadratic equations. Further, with the help of the non-convex optimization theory, the problem to identify a feasible solution is formulated into a non-convex optimization problem.
	
	\subsection{Some Properties and An Alternative Expression}\label{non-convex-form}
	Following the above discussion, for each real root $\lambda_i$ of $q(s)$, let $J^{(i)}(\lambda_i)$ be the corresponding Jordan block for $1\leq i\leq \ell$,  and for each pair of conjugate complex roots $\mu_{j}\pm \imath\omega_{j}$, let $J^{(j)}(\mu_{j},\omega_{j})$ be the corresponding Jordan block for $\ell+1\leq j\leq k$. Then
	\begin{align}\label{prop-1-10}
		J={\sf diag}
		\Big(J^{(1)}(\lambda_1),\cdots,J^{(\ell)}(\lambda_{\ell}),  J^{(\ell+1)}(\mu_{\ell+1},\omega_{\ell+1}),\cdots,J^{(k)}(\mu_{k},\omega_{k})\Big).
	\end{align}
	For any positive integer $m$, let
	\begin{align*}
		{\cal J}_m(\lambda)=\left[
		\begin{array}{cccc}
			\lambda& & & \\
			1 &\ddots & & \\
			&\ddots &\ddots & \\
			& &1 &\lambda
		\end{array}
		\right]_{m\times m};
	\end{align*}
	\begin{align*}
		{\cal J}_m(\mu,\omega)=\left[
		\begin{array}{cccc}
			\Theta(\mu,\omega)& & & \\
			I &\ddots & & \\
			&\ddots &\ddots & \\
			& &I &\Theta(\mu,\omega)
		\end{array}
		\right]_{(2m)\times (2m)} \mbox{with} \ \Theta(\mu,\omega)=\begin{bmatrix}
			\mu&-\omega\\
			\omega& \mu
		\end{bmatrix}.
	\end{align*}
	Due to the multiplicity, the form of $J$ is generally not unique. If we consider $J$ with (\ref{equ_Lapl}) together, however, we are able to prove the uniqueness of the Jordan form $J$ in the following proposition, whose proof is postponed to the Appendix.
	\begin{proposition}\label{prop_Jordan}
		For rational function $ {\cal L}(s)$ given by \eqref{prob1}, suppose that there exists an $n$-order PH representation $(\alpha, A)$ such that \eqref{equ_Lapl} holds. Then each distinct eigenvalue of $A$ has exactly one Jordan block, that is, for $1\leq i\leq \ell$, $J^{(i)}(\lambda_i)={\cal J}_{n_i}(\lambda_i)$, and
		for $\ell+1\leq j\leq k$,  $J^{(j)}(\mu_j,\omega_{j})={\cal J}_{n_j}(\mu_j,\omega_{j})$.
		In other words, the Jordan block has a unique form
		\begin{align}\label{prop-1-1}
			{\sf diag} \Big(
			{\cal J}_{n_1}(\lambda_1),\cdots,{\cal J}_{n_\ell}(\lambda_{\ell}),{\cal J}_{n_{\ell+1}}(\mu_{\ell+1},\omega_{\ell+1}),\cdots,{\cal J}_{n_k}(\mu_{k},\omega_{k})  \Big),
		\end{align}
		which is denoted by ${\cal J}$ in the following analysis.
	\end{proposition}
	
	Next we provide a property of $P$ which helps us to further simplify \eqref{equ_Jordanform}, and whose proof is given in the Appendix.
	\begin{proposition}\label{prop_P}
		For rational function ${\cal L}(s)$ given by \eqref{prob1}, suppose that there exists an $n$-order PH representation $(\alpha,A)$ such that \eqref{equ_Lapl} holds. Then there exists a nonsingular matrix $P$ such that $A=P^{-1}{\cal J}P$ and $P{\bf 1}^\top={\bf 1}^\top$, where
		${\cal J}$ is given by \eqref{prop-1-1} in Proposition {\rm \ref{prop_Jordan}}.
	\end{proposition}
	
	%

	Now we do partial-fraction expansion for ${\cal L}(s)$ according to the factorization of $q(s)$ given by
	(\ref{equ_q}). That is,
	\begin{align*}
		{\cal L}(s)=\sum_{i=1}^\ell \sum_{r=1}^{n_i}\frac{ c_{i,r}}{(s-\lambda_i)^r}+\sum_{j=\ell+1}^k \sum_{r=1}^{n_j} \big(\frac{c_{j,r}}{(s-\mu_j-\imath\omega_j)^r}+\frac{ \overline{c}_{j,r}}{(s-\mu_j+\imath\omega_j)^r}\big),
	\end{align*}
	where $ c_{i,r}\in {\cal R}$ and $ c_{i,n_i}\neq 0$ for $1\leq r\leq n_i$ and $1\leq i\leq\ell$, $ c_{j,r}\in {\cal C}$, $ \overline{c}_{j,r}$ is the conjugate of $ c_{j,r}$, and $ c_{j,n_j}\neq 0$ for $1\leq r\leq n_j$ and $\ell+1\leq j\leq k$. Let $\beta$ be the solution to
	\begin{align}
		&\sum_{i=1}^\ell \sum_{r=1}^{n_i}\frac{ c_{i,r}}{(s-\lambda_i)^r}+\sum_{j=\ell+1}^k \sum_{r=1}^{n_j} \big(\frac{ c_{j,r}}{(s-\mu_j-\imath\omega_j)^r}+\frac{ \overline{c}_{j,r}}{(s-\mu_j+\imath\omega_j)^r}\big) \nonumber\\
		& \ \ \ =-\beta {\cal J}(sI-{\cal J})^{-1}{\bf 1}^\top,\label{equ_beta}
	\end{align}
	where ${\cal J}$ is given by (\ref{prop-1-1}) in Proposition \ref{prop_Jordan}. To be consistent with the block structure of ${\cal J}$,
	we take the blocks of $\beta$ into the following way:
	\begin{align*}
		\beta &=\Big(\beta(\lambda_1),\ldots,\beta(\lambda_\ell),\beta(\mu_{\ell+1},\omega_{\ell+1}),\ldots,\beta(\mu_k,\omega_k)\Big);\\
		\beta(\lambda_i) &=\Big(\beta_1(\lambda_i),\ldots, \beta_{n_i}(\lambda_i)\Big) \ \mbox{for $1\leq i\leq \ell$};\\
		\beta(\mu_{j},\omega_j) &=\Big(\beta^{(1)}_1(\mu_{j},\omega_j),\beta^{(1)}_2(\mu_{j},\omega_j),\ldots,\beta^{(n_j)}_1(\mu_{j},\omega_j),\beta^{(n_j)}_2(\mu_{j},\omega_j)
		\Big) \ \mbox{for $\ell+1\leq j\leq k$.}
	\end{align*}
	The following proposition shows that $\beta$ is well-defined.
	\begin{proposition}\label{prop_beta}
		For rational function ${\cal L}(s)$ given by \eqref{prob1},  $\beta$ is well defined by \eqref{equ_beta}. Further,
		$\beta$ is an $n$-dimensional real vector, 	$\beta_{n_i}(\lambda_i)\neq 0$ for $1\leq i\leq \ell$ , and $\big(\beta^{(n_j)}_{1}(\mu_j,\omega_j)\big)^2$ $+\big(\beta^{(n_j)}_2(\mu_j,\omega_j)\big)^2\neq 0$  for $\ell+1 \leq j\leq k$.
	\end{proposition}
	
	\begin{proof}  Note that
		\begin{eqnarray}
			&& \Big(sI-{\cal J}_{m}(\lambda)\Big)^{-1}=\begin{bmatrix}
				(s-\lambda)^{-1}& & & \\
				(s-\lambda)^{-2}&(s-\lambda)^{-1}& &\\
				\vdots&\ddots &\ddots &\\
				(s-\lambda)^{-m}&\cdots&(s-\lambda)^{-2}&(s-\lambda)^{-1}
			\end{bmatrix},\label{inverse-J-1}\\
			&& \Big(sI-{\cal J}_{m}(\mu,\omega)\Big)^{-1}\label{inverse-J-2}\\
			&& \ \ =
			\begin{bmatrix}
				(sI-\Theta(\mu,\omega))^{-1}& & & \\
				(sI-\Theta(\mu,\omega))^{-2}&(sI-\Theta(\mu,\omega))^{-1}& &\\
				\vdots&\ddots &\ddots &\\
				(sI-\Theta(\mu,\omega))^{-m}&\cdots&(sI-\Theta(\mu,\omega))^{-2}&(sI-\Theta(\mu,\omega))^{-1}
			\end{bmatrix},
			\nonumber
		\end{eqnarray}
		where for any positive integer $m$, $(sI-\Theta(\mu,\omega))^{-m}=\Big((sI-\Theta(\mu,\omega))^{-1}\Big)^m$ and
		\begin{align*}
			(sI-\Theta(\mu,\omega))^{-1}=\begin{bmatrix}
				s-\mu &\omega\\
				-\omega& s-\mu
			\end{bmatrix}^{-1}=\frac{1}{(s-\mu)^2+\omega^2}\begin{bmatrix}
				s-\mu &-\omega\\
				\omega& s-\mu
			\end{bmatrix}.
		\end{align*}
		We first prove the part corresponding to the real pole $\lambda_i$.
		By (\ref{prop-1-1}) and (\ref{inverse-J-1}),
		\[
		{\cal J}_{n_i}(\lambda_i)(sI- {\cal J}_{n_i}(\lambda_i))^{-1}{\bf 1}^\top=
		\left(
		\begin{array}{ccc}
			\frac{\lambda_i}{s-\lambda_i}\\
			\frac{\lambda_i}{(s-\lambda_i)^2}+\frac{\lambda_i+1}{s-\lambda_i}\\
			\vdots\\
			\frac{\lambda_i}{(s-\lambda_i)^{n_i}}+\sum_{r=1}^{n_i-1}\frac{\lambda_i+1}{(s-\lambda_i)^r}
		\end{array}
		\right).
		\]
		Then $\beta(\lambda_i)=\big(\beta_{1}(\lambda_i),\ldots,\beta_{n_i}(\lambda_i)\big)$ can be uniquely determined by (\ref{equ_beta}) as follows:
		\begin{eqnarray}
			\beta_{n_i}(\lambda_i)&=&-\frac{ c_{i,n_i}}{\lambda_i};\label{prop-2-7}\\
			\beta_{r}(\lambda_i)&=&-\frac{1}{\lambda_i}\Big( c_{i,r}+(\lambda_i+1)\sum_{l=1}^{n_i-r}\Big(-\frac{1}{\lambda_i}\Big)^l c_{i,r+l}\Big) \
			\mbox{for } \ 1\leq r\leq n_i-1. \label{prop-2-8}
		\end{eqnarray}
		The proposition for $\beta(\lambda_i)$ with $1\leq i\leq \ell$ follows from (\ref{prop-2-7})-(\ref{prop-2-8}).
		
		Now consider the proposition corresponding to the pair of conjugate complex poles $\mu_{j}\pm \imath\omega_{j}$
		with $\ell+1\leq j\leq k$.
		Let
		\begin{small}
			\begin{align*}
				U_{n_j}(\imath,-\imath)=\frac{1}{\sqrt{2}}\begin{bmatrix}
					1&\imath& & & & \\
					& & \cdots&\cdots& & \\
					& & & & 1&\imath\\
					1&-\imath& & & & \\
					& &\cdots&\cdots& & \\
					& & & & 1&-\imath
				\end{bmatrix}_{2n_j\times2n_j},
			\end{align*}
		\end{small}
		and  $U_{n_j}^ H(\imath,-\imath)$ denote its conjugate transpose.	Note that $U_{n_j}(\imath,-\imath)U_{n_j}^ H(\imath,-\imath)=I$ and
		\begin{align}\label{u-j}
			{\cal J}_{n_j}(\mu_j,\omega_j)=U_{n_j}^H(\imath,-\imath)  \widetilde{\cal J}_{n_j}(\mu_j,\omega_j)U_{n_j}(\imath,-\imath),
		\end{align}
		where
		$
		\widetilde{\cal J}_{n_j}(\mu_j,\omega_j)=\begin{bmatrix}
			\widetilde{\cal J}_{n_j}(\mu_j+\imath\omega_j)&\\
			& \widetilde{\cal J}_{n_j}(\mu_j-\imath\omega_j)
		\end{bmatrix}
		$ 
		with $ \widetilde{\cal J}_{n_j}(\mu_j+\imath\omega_j)$ and $\widetilde{\cal J}_{n_j}(\mu_j-\imath\omega_j)$ being given by
		\begin{small}
			\begin{align*}
				\begin{bmatrix}
					\mu_j+\imath\omega_j& & & \\
					1 &\ddots & & \\
					&\ddots &\ddots &\\
					& &1 &\mu_j+\imath\omega_j
				\end{bmatrix}_{n_j\times n_j} \ \mbox{and} \
				\begin{bmatrix}
					\mu_j-\imath\omega_j& & & \\
					1 &\ddots & & \\
					&\ddots &\ddots &\\
					& &1 &\mu_j-\imath\omega_j
				\end{bmatrix}_{n_j\times n_j},
			\end{align*}
		\end{small}
		respectively.  From \eqref{u-j},
		\begin{align*}
			&\beta(\mu_j,\omega_j) {\cal J}_{n_j}(\mu_j,\omega_j)\Big(sI- {\cal J}_{n_j}(\mu_j,\omega_j)\Big)^{-1}{\bf 1}^\top\\
			& \ \ \ =\beta(\mu_j,\omega_j)U_{n_j}^ H(\imath,-\imath)  \widetilde{\cal J}_{n_j}(\mu_j,\omega_j)\Big(sI- \widetilde{\cal J}_{n_j}(\mu_j,\omega_j)\Big)^{-1}U_{n_j}(\imath,-\imath){\bf 1}^\top.
		\end{align*}
		Observe that
		\begin{align*}
			U_{n_j}(\imath,-\imath){\bf 1}^\top=\frac{1}{\sqrt{2}}\Big(\underbrace{1+\imath,1+\imath,\cdots, 1+\imath}_{n_j},\underbrace{1-\imath,1-\imath,\cdots, 1-\imath}_{n_j}\Big)^\top.
		\end{align*}
		Then
		\begin{align*}
			& \widetilde{\cal J}_{n_j}(\mu_j,\omega_j)\Big(sI- \widetilde{\cal J}_{n_j}(\mu_j,\omega_j)\Big)^{-1}U_{n_j}(\imath,-\imath){\bf 1}^\top\\
			& \ \ \ =\left(
			\begin{array}{cc}
				\frac{1+\imath}{\sqrt{2}}\widetilde{\cal J}_{n_j}(\mu_j+\imath\omega_j)\Big(sI- \widetilde{\cal J}_{n_j}(\mu_j+\imath\omega_j)\Big)^{-1}{\bf 1}^\top\\
				\frac{1-\imath}{\sqrt{2}}\widetilde{\cal J}_{n_j}(\mu_j-\imath\omega_j)\Big(sI- \widetilde{\cal J}_{n_j}(\mu_j-\imath\omega_j)\Big)^{-1}{\bf 1}^\top
			\end{array}
			\right).
		\end{align*}
		Similar to the case of real poles, by (\ref{inverse-J-2}), we have
		\begin{align*}
			\widetilde{\cal J}_{n_j}(\mu_j+\imath\omega_j)\Big(sI- \widetilde{\cal J}_{n_j}(\mu_j+\imath\omega_j)\Big)^{-1}{\bf 1}^\top=\left(
			\begin{array}{ccc}
				\frac{\mu_j+ \imath\omega_j}{s-(\mu_j+\imath\omega_j)}\\
				\frac{\mu_j+\imath\omega_j}{(s-(\mu_j+\imath\omega_j))^2}+\frac{1+\mu_j+\imath\omega_j}{s-(\mu_j+\imath\omega_j)}\\
				\vdots\\
				\frac{\mu_j+\imath\omega_j}{(s-(\mu_j+\imath\omega_j))^{n_j}}+\sum_{r=1}^{n_j-1}\frac{1+\mu_j+\imath\omega_j}{(s-(\mu_j+\imath\omega_j))^r}
			\end{array}
			\right).
		\end{align*}
		and
		\begin{align*}
			\widetilde{\cal J}_{n_j}(\mu_j-\imath\omega_j)\Big(sI- \widetilde{\cal J}_{n_j}(\mu_j-\imath\omega_j)\Big)^{-1}{\bf 1}^\top=\left(
			\begin{array}{ccc}
				\frac{\mu_j- \imath\omega_j}{s-(\mu_j-\imath\omega_j)}\\
				\frac{\mu_j-\imath\omega_j}{(s-(\mu_j-\imath\omega_j))^2}+\frac{1+\mu_j-\imath\omega_j}{s-(\mu_j-\imath\omega_j)}\\
				\vdots\\
				\frac{\mu_j-\imath\omega_j}{(s-(\mu_j-\imath\omega_j))^{n_j}}+\sum_{r=1}^{n_j-1}\frac{1+\mu_j-\imath\omega_j}{(s-(\mu_j-\imath\omega_j))^r}
			\end{array}
			\right).
		\end{align*}
		Note that
		\begin{align*}
			\beta(\mu_j,\omega_j)U_{n_j}^ H(\imath,-\imath)=&\Big(\beta^{(1)}_1(\mu_j,\omega_j)-\imath \beta^{(1)}_2(\mu_j,\omega_j),\beta^{(2)}_1(\mu_j,\omega_j)-\imath \beta^{(2)}_2(\mu_j,\omega_j),\ldots,\\
			&\beta^{(n_j)}_1(\mu_j,\omega_j)-\imath \beta^{(n_j)}_2(\mu_j,\omega_j),
			\beta^{(1)}_1(\mu_j,\omega_j)+\imath \beta^{(1)}_2(\mu_j,\omega_j),\\
			&\beta^{(2)}_1(\mu_j,\omega_j)+\imath \beta^{(2)}_2(\mu_j,\omega_j),\ldots,\beta^{(n_j)}_1(\mu_j,\omega_j)+\imath \beta^{(n_j)}_2(\mu_j,\omega_j)\Big).
		\end{align*}
		To prove the proposition, by (\ref{prop-1-1}) and with a consideration of (\ref{prop-2-7})-(\ref{prop-2-8}) for $1\leq i\leq \ell$, it suffices to prove that
		for $j=\ell+1,\ldots,k$,
		\[
		\beta(\mu_j,\omega_j)=\Big(\beta^{(1)}_1(\mu_j,\omega_j),\beta^{(1)}_2(\mu_j,\omega_j),\ldots,
		\beta^{(n_j)}_1(\mu_j,\omega_j),\beta^{(n_j)}_2(\mu_j,\omega_j)\Big)
		\]
		is uniquely determined by the following equation:
		\begin{eqnarray*}
			&&\sum_{r=1}^{n_j} \big(\frac{ c_{j,r}}{(s-\mu_j-\imath\omega_j)^r}+\frac{ \overline{c}_{j,r}}{(s-\mu_j+\imath\omega_j)^r}\big)\nonumber\\
			&& \ \ \ = -\beta(\mu_j,\omega_j)U_{n_j}^ H(\imath,-\imath) 	\left(\begin{array}{cc}
				\frac{1+\imath}{\sqrt{2}} \widetilde{\cal J}_{n_j}(\mu_j+\imath\omega_j)\Big(sI- \widetilde{\cal J}_{n_j}(\mu_j+\imath\omega_j)\Big)^{-1}{\bf 1}^\top\\
				\frac{1-\imath}{\sqrt{2}} \widetilde{\cal J}_{n_j}(\mu_j-\imath\omega_j)\Big(sI- \widetilde{\cal J}_{n_j}(\mu_j-\imath\omega_j)\Big)^{-1}{\bf 1}^\top
			\end{array}
			\right).
		\end{eqnarray*}
		The two terms on the left hand side of the above equation can be expressed as
		\begin{align*}
			\sum_{r=1}^{n_j} \frac{ c_{j,r}}{(s-\mu_j-\imath\omega_j)^r}\nonumber=& - \Big(\beta^{(1)}_1(\mu_j,\omega_j)-\imath \beta^{(1)}_2(\mu_j,\omega_j),\beta^{(2)}_1(\mu_j,\omega_j)-\imath \beta^{(2)}_2(\mu_j,\omega_j),\ldots,\\
			&\quad \quad \beta^{(n_j)}_1(\mu_j,\omega_j)-\imath \beta^{(n_j)}_2(\mu_j,\omega_j)	\Big)\\
			&\quad \quad \times
			\frac{1+\imath}{\sqrt{2}}\widetilde{\cal J}_{n_j}(\mu_j+\imath\omega_j)\Big(sI- \widetilde{\cal J}_{n_j}(\mu_j+\imath\omega_j)\Big)^{-1}{\bf 1}^\top;\\
			\sum_{r=1}^{n_j} \frac{ \overline{c}_{j,r}}{(s-\mu_j+\imath\omega_j)^r}=& - \Big(\beta^{(1)}_1(\mu_j,\omega_j)+\imath \beta^{(1)}_2(\mu_j,\omega_j),\beta^{(2)}_1(\mu_j,\omega_j)+\imath \beta^{(2)}_2(\mu_j,\omega_j),\ldots,\\
			&\quad \quad \beta^{(n_j)}_1(\mu_j,\omega_j)+\imath \beta^{(n_j)}_2(\mu_j,\omega_j)	\Big)\\
			&\quad \quad \times
			\frac{1-\imath}{\sqrt{2}} \widetilde{\cal J}_{n_j}(\mu_j-\imath\omega_j)\Big(sI- \widetilde{\cal J}_{n_j}(\mu_j-\imath\omega_j)\Big)^{-1}{\bf 1}^\top.
		\end{align*}
		Thus, we have
		\begin{align*}
			&\beta^{(n_j)}_1(\mu_j,\omega_j)-\imath \beta^{(n_j)}_2(\mu_j,\omega_j)=-\frac{(1-\imath) c_{j,n_j}}{\sqrt{2}(\mu_j+\imath\omega_j)};\\
			&\beta^{(r)}_1(\mu_j,\omega_j)-\imath \beta^{(r)}_2(\mu_j,\omega_j)\nonumber\\
			&\quad  =-\frac{1-\imath}{\sqrt{2}(\mu_j+\imath\omega_j)}\Big( c_{j,r}+(1+\mu_j+\imath\omega_j)\sum_{l=1}^{n_j-r}\Big(-\frac{1}{\mu_j+\imath\omega_j}\Big)^l c_{j,r+l}\Big) \ \mbox{for} \ 1\leq r\leq n_j-1;\\
			&\beta^{(n_j)}_1(\mu_j,\omega_j)+\imath \beta^{(n_j)}_2(\mu_j,\omega_j)=-\frac{(1+\imath) \overline{c}_{j,n_j}}{\sqrt{2}(\mu_j-\imath\omega_j)};\\
			&\beta^{(r)}_1(\mu_j,\omega_j)+\imath \beta^{(r)}_2(\mu_j,\omega_j)\\
			&\quad =-\frac{1+\imath}{\sqrt{2}(\mu_j-\imath\omega_j)}\Big( \overline{c}_{j,r}+(1+\mu_j-\imath\omega_j)\sum_{l=1}^{n_j-r}\Big(-\frac{1}{\mu_j-\imath\omega_j}
			\Big)^l \overline{c}_{j,r+l}\Big) \ \mbox{for} \ 1\leq r\leq n_j-1.
		\end{align*}
		Therefore, $\beta(\mu_j,\omega_j)$ can be uniquely determined as follows: for  $ 1\leq r\leq n_j-1$,
		\begin{align}
			\beta^{(r)}_1(\mu_j,\omega_j)&=-{\sf Re}\bigg(\frac{1-\imath}{\sqrt{2}(\mu_j+\imath\omega_j)}\Big( c_{j,r}+(1+\mu_j+\imath\omega_j)\sum_{l=1}^{n_j-r}\Big(-\frac{1}{\mu_j+\imath\omega_j}\Big)^l c_{j,r+l}\Big)\bigg);
			\label{prop-2-10}\\
			\beta^{(r)}_2(\mu_j,\omega_j)&={\sf Im}\bigg(\frac{1-\imath}{\sqrt{2}(\mu_j+\imath\omega_j)}\Big( c_{j,r}+(1+\mu_j+\imath\omega_j)\sum_{l=1}^{n_j-r}\Big(-\frac{1}{\mu_j+\imath\omega_j}\Big)^l c_{j,r+l}\Big)\bigg);
			\label{prop-2-11}
		\end{align}
		\begin{align}
			\beta^{(n_j)}_1(\mu_j,\omega_j)&=-{\sf Re}\Big(\frac{(1-\imath) c_{j,n_j}}{\sqrt{2}(\mu_j+\imath\omega_j)}\Big);\label{prop-2-12}\\
			\beta^{(n_j)}_2(\mu_j,\omega_j)&={\sf Im}\Big(\frac{(1-\imath) c_{j,n_j}}{\sqrt{2}(\mu_j+\imath\omega_j)}\Big).\label{prop-2-13}
		\end{align}
		
		The proposition corresponding to the pair of conjugate complex poles $\mu_{j}\pm \imath\omega_{j}$
		with $\ell+1\leq j\leq k$ follows from (\ref{prop-2-10})-(\ref{prop-2-13}).  Therefore, $\beta$ is well defined by \eqref{equ_beta} and is an $n$-dimensional real vector.
		
		Now we show that $\beta_{n_i}(\lambda_i)\neq 0$ for $1\leq i\leq \ell$ , and $\big(\beta^{(n_j)}_{1}(\mu_j,\omega_j)\big)^2$ $+\big(\beta^{(n_j)}_2(\mu_j,\omega_j)\big)^2$ $\neq 0$  for $\ell+1 \leq j\leq k$. Note that
		\begin{small}
			\begin{align*}
				&\beta {\cal J}(sI-{\cal J})^{-1}{\bf 1}^\top=	\beta (sI-{\cal J})^{-1}{\cal J}{\bf 1}^\top\\
				&\quad= \bigg(\beta(\lambda_1)\big(sI-{\cal J}_{n_1}(\lambda_1)\big)^{-1},...,\beta(\lambda_{\ell})\big(sI- {\cal J}_{n_{\ell}}(\lambda_{\ell})\big)^{-1},\\
				&\quad\quad\quad\beta(\mu_{\ell+1},\omega_{\ell+1})\big(sI- {\cal J}_{n_{\ell+1}}(\mu_{\ell+1},\omega_{\ell+1})\big)^{-1},..., \beta(\mu_k,\omega_k)\big(sI- {\cal J}_{n_k}(\mu_k,\omega_k)\big)^{-1}\bigg){\cal J}{\bf 1}^\top.
			\end{align*}
		\end{small}
		Recall \eqref{inverse-J-1} and \eqref{inverse-J-2}, $(s-\lambda_i)^{-n_i}$ only appears in the $n_i$-th row of $(sI- {\cal J}_{n_i}(\lambda_i))^{-1}$ and $((s-\mu_j)^2+\omega_j^2)^{-n_j}$ only appears in the $(2n_j-1)$-th and $2n_j$-th rows of $(sI-{\cal J}_{n_j}(\mu_j\pm \imath \omega_j))^{-1}$. By \eqref{equ_beta}, since $ c_{i,n_i}\neq 0$ and $ c_{j,n_j}\neq 0$, we have $\beta_{n_i}(\lambda_i)\neq 0$ for $1\leq i\leq \ell$ , and $\big(\beta^{(n_j)}_{1}(\mu_j,\omega_j)\big)^2$ $+\big(\beta^{(n_j)}_2(\mu_j,\omega_j)\big)^2$ $\neq 0$  for $\ell+1 \leq j\leq k$. Hence, the proof is complete.
	\end{proof}
	
	For any rational function $ {\cal L}(s)$ given by \eqref{prob1},  from its poles given in (\ref{pole}), we can determine
	$ {\cal J}_{n_i}(\lambda_i)$ for $1\leq i\leq \ell$, and $ {\cal J}_{n_j}(\mu_j,\omega_j)$ for $\ell+1\leq j\leq k$, then $J$ can be constructed as
	(\ref{prop-1-1}). Further, the $n$-dimensional vector $\beta$ can be determined in (\ref{equ_beta}) from Proposition  \ref{prop_beta}.
	With the help of Propositions \ref{prop_Jordan}-\ref{prop_P}, we can transform the question of finding an $n$-order PH representation $(\alpha,A)$ such that \eqref{equ_Lapl}  holds into the problem of identifying a feasible solution of the set of linear  and quadratic equations.
	To that end, let
	\begin{align}\label{equ:xi}
		\xi=-\Big(\sum_{i=1}^\ell n_i\lambda_i+\sum_{j=\ell+1}^k2n_j\mu_j\Big).
	\end{align}
	For the sake of the notation simplicity, for any matrix $H=(h_{ij})_{n\times n}\in {\cal R}^{n\times n}$, its following row-blocked form is called {\it ${\cal J}$-blocked  partition} with respect to the block structure of ${\cal J}$ given by (\ref{prop-1-1}):
	\begin{eqnarray*}
		H=\left[
		\begin{array}{cccc}
			H(\lambda_1)\\
			\vdots\\
			H(\lambda_\ell)\\
			H(\mu_{\ell+1},\omega_{\ell+1})\\
			\vdots\\
			H(\mu_k,\omega_k)
		\end{array}
		\right] \ \mbox{with} \
		H(\lambda_i)=\left[
		\begin{array}{cccc}
			H_1(\lambda_i)\\
			\vdots\\
			H_{n_i}(\lambda_i)
		\end{array}
		\right]
	\end{eqnarray*}
	\begin{eqnarray*}
		\mbox{and} \
		H(\mu_j,\omega_j)=\left[
		\begin{array}{cccc}
			H^{(1)}(\mu_j,\omega_j)\\
			\vdots\\
			H^{(n_j)}(\mu_j,\omega_j)
		\end{array}
		\right], \
		H^{(r)}(\mu_j,\omega_j)=\left[
		\begin{array}{cc}
			H^{(r)}_1(\mu_j,\omega_j)\\
			H^{(r)}_2(\mu_j,\omega_j)
		\end{array}
		\right], \ 1\leq r\leq n_j,
	\end{eqnarray*}
	where $H(\lambda_1)$ is the first $n_i$ rows of $H$, and in general, $H(\lambda_i)$ is the rows of $H$ through
	$(n_{[0,i-1]}+1)$ to $n_{[0,i]}$; and
	$H(\mu_j,\omega_j)$ is the rows of $H$ through $(n_{[0,j-1]}+1)$ to
	$n_{[0, j]}$ with the convention
	\[
	n_{[0,r]}=\left\{
	\begin{array}{lll}
		0, &\mbox{if $r=0$};\\
		n_1+\cdots+n_r, &\mbox{if $1\leq r\leq \ell$};\\
		n_1+\cdots+n_\ell+2(n_{\ell+1}+\cdots+n_{r}), &\mbox{if $\ell+1\leq r\leq k$}.
	\end{array}
	\right.
	\]
	Clearly,
	$H(\lambda_i)\in {\cal R}^{n_i\times n}$ for $1\leq i\leq \ell$, and $H(\mu_j,\omega_j)\in {\cal R}^{2n_j\times n}$
	for $\ell+1\leq j\leq k$.

	\begin{theorem}\label{iff}
		For rational function ${\cal L}(s)$ given by \eqref{prob1}, there exists an $n$-order continuous-time PH representation $(\alpha,A)$ such that \eqref{equ_Lapl} holds if and only if there exist $P=(p_{ij})_{n\times n}$ and $A=(a_{ij})_{n\times n}$ such that the following $($in$)$equalities hold$:$
		\begin{eqnarray}
			\left\{
			\begin{array}{lllll}
				&PA={\cal J}P;\\
				& P {\bf 1}^\top={\bf 1}^\top;\\
				&\beta  P\geq {\bf 0};\\
				&A {\bf 1}^\top \leq {\bf 0};\\
				& -\xi\leq a_{ii}\leq 0 \ \mbox{for} \ 1\leq i\leq n;\\
				&0\leq a_{ij}\leq \xi \ \mbox{for} \ i\neq j, 1\leq i,j\leq n,
			\end{array}
			\right. \label{equ_sys}
		\end{eqnarray}	
		where ${\cal J}$ is given by \eqref{prop-1-1} in Proposition {\rm \ref{prop_Jordan}}, and $\beta$ is given by Proposition {\rm \ref{prop_beta}}. Moreover, the feasible solution $(A,P)$ to \eqref{equ_sys} directly provides an $n$-order continuous-time PH representation $(\alpha,A)$ with
		$\alpha=\beta P$.
	\end{theorem}
	
	\begin{proof}
		First we prove the necessity.	From Proposition \ref{prop_P}, there exists nonsingular $P$ such that the first two equations hold in (\ref{equ_sys}).  That is, $A=P^{-1}{\cal J}P$ and $P {\bf 1}^\top={\bf 1}^\top$. Thus, from \eqref{equ_Lapl},
		\begin{align*}
			{\cal L}(s)=\frac{p(s)}{q(s)}=-\alpha A(sI-A)^{-1}{\bf 1}^\top=-\alpha P^{-1} {\cal J}(sI-{\cal J})^{-1}{\bf 1}^\top.
		\end{align*}
		Then the uniqueness of $\beta$ in Proposition \ref{prop_beta} implies that $\alpha P^{-1}=\beta$. This gives that $\beta  P=\alpha\geq {\bf 0}$. Again from $PA={\cal J}P$ and nonsingularity of $P$, the spectrum of $A$ is the same as that of ${\cal J}$.
		Then
		\begin{align*}
			\sum_{i=1}^na_{ii}={\sf tr}(A)=\sum_{i=1}^\ell n_i\lambda_i+\sum_{j=\ell+1}^k2n_j\mu_j= -\xi.
		\end{align*}
		As $a_{ii}\leq0$ for each $i$, we have $-\xi\leq a_{ii}\leq 0$ for $1\leq i\leq n$. By $A {\bf 1}^\top \leq {\bf 0}$ and $ a_{ij}\geq 0$ for $i\neq j$, we obtain that $a_{ij}\leq-a_{ii}\leq  \xi$ for $i\neq j, 1\leq i,j\leq n$.
		
		Now consider the sufficiency. Make ${\cal J}$-blocked partition for matrix $P$,
		\begin{align}
			P=\left[\begin{array}{ccc}
				P(\lambda_1)\\
				\vdots\\
				P(\lambda_\ell)\\
				P(\mu_{\ell+1},\omega_{\ell+1})\\
				\vdots\\
				P(\mu_{k},\omega_k)\\
			\end{array}\right].\label{P-block}
		\end{align}
		From the first relation in \eqref{equ_sys}, we have $P(\lambda_i)A= {\cal J}_{n_i}(\lambda_i)P(\lambda_i)$ and $P(\mu_{j}, \omega_j)A= {\cal J}_{n_j}(\mu_{j},\omega_j)P(\mu_{j},\omega_j)$. Recall the form of $ {\cal J}_{n_i}(\lambda_i)$, we have
		\begin{align}\label{real_P}
			\begin{split}
				&P_{1}(\lambda_i)(A-\lambda_iI)=0;\\
				&P_{r}(\lambda_i)(A-\lambda_iI)=P_{r-1}(\lambda_i), \quad 2\leq r\leq n_i.
			\end{split}
		\end{align}
			Recall the structure of ${\cal J}_{n_j}(\mu_{j}, \omega_j)$. We have
			\begin{align*}
				&P^{(1)}_1(\mu_{j},\omega_j)A=\mu_jP^{(1)}_1(\mu_{j},\omega_j)-\omega_jP^{(1)}_2(\mu_{j},\omega_j),\\
				&P^{(1)}_2(\mu_{j},\omega_j)A=\mu_jP^{(1)}_2(\mu_{j},\omega_j)+\omega_jP^{(1)}_1(\mu_{j},\omega_j),\\
				&P^{(r)}_1(\mu_{j},\omega_j)A=P^{(r-1)}_1(\mu_{j},\omega_j)+\mu_jP^{(r)}_1(\mu_{j},\omega_j)-\omega_jP^{(r)}_2(\mu_{j},\omega_j),\quad 2\leq r\leq n_j,\\
				&P^{(r)}_2(\mu_{j},\omega_j)A=P^{(r-1)}_2(\mu_{j},\omega_j)+\mu_jP^{(r)}_2(\mu_{j},\omega_j)+\omega_jP^{(r)}_1(\mu_{j},\omega_j),\quad 2\leq r\leq n_j.
			\end{align*}
			Equivalently,
			\begin{align}\label{complex_P}
				\begin{split}
					&\Big(P^{(1)}_1(\mu_{j},\omega_j)+\imath P^{(1)}_2(\mu_{j},\omega_j)\Big)\Big(A-(\mu_j+\imath \omega_j)I\Big)=0,\\
					&\Big(P^{(1)}_1(\mu_{j},\omega_j)-\imath P^{(1)}_2(\mu_{j},\omega_j)\Big)\Big(A-(\mu_j-\imath \omega_j)I\Big)=0,\\
					&\Big(P^{(r)}_1(\mu_{j},\omega_j)+\imath P^{(r)}_2(\mu_{j},\omega_j)\Big)\Big(A-(\mu_j+\imath \omega_j)I\Big)\\
					&\  =P^{(r-1)}_1(\mu_{j},\omega_j)+\imath P^{(r-1)}_2(\mu_{j},\omega_j),\quad 2\leq r\leq n_j,\\
					&\Big(P^{(r)}_1(\mu_{j},\omega_j)-\imath P^{(r)}_2(\mu_{j},\omega_j)\Big)\Big(A-(\mu_j-\imath \omega_j)I\Big)\\
					&\ =P^{(r-1)}_1(\mu_{j},\omega_j) -\imath P^{(r-1)}_2(\mu_{j},\omega_j),\quad 2\leq r\leq n_j.
				\end{split}
			\end{align}		
			It follows from the second equation in \eqref{equ_sys} that each row vector of $P$ is not zero. By Exercise in Section 6.3 of \cite{lancaster1985theory}, \eqref{real_P} and \eqref{complex_P} imply that $P$ is the (generalized) eigenvector matrix of $A$ so it is nonsingular.	Then the nonsingularity of ${\cal J}$ gives that $A$ is nonsingular. Let
			$\alpha=\beta P$. By the definition of $\beta$, we have
			\begin{align*}
				{\cal L}(s)=-\beta {\cal J}(sI-{\cal J})^{-1}{\bf 1}^\top=-\beta PP^{-1}{\cal J}(sI-{\cal J})^{-1}P{\bf 1}^\top=-\alpha A(sI-A)^{-1}{\bf 1}^\top.
			\end{align*}
			It remains to show that $\alpha{\bf 1}^\top=1$. This is straightforward by ${\cal L}(0)=1$.
			Therefore, the sufficiency is proved.
		\end{proof}
		
		With the help of Propositions \ref{prop_Jordan}-\ref{prop_beta}, Theorem \ref{iff} equivalently transforms the problem of finding the minimal representation $(\alpha, A)$ with the corresponding LST (${\cal L}(s)$)  given by (\ref{prob1}) into the solution existence problem for a set of linear and quadratic equations given by (\ref{equ_sys}). The transformation is based on vector $\beta$ and Jordan form ${\cal J}$. Both of them are uniquely determined from ${\cal L}(s)$ although ${\cal J}$ is the Jordan form of $A$. Telek and Horv\'ath \cite{telek2007} also provide a representation transformation but by considering the Jordan form of $(-A)^{-1}$ rather than $A$ itself. We know that the algorithm involving the computation of the matrix inverse usually leads  to numerical instability and high computational complexity.
		Compared with Telek and Horv\'ath \cite{telek2007}, moreover, the properties of $\beta$ obtained in Proposition \ref{prop_beta} can help us to establish the convergence of the {\bf AM} algorithm proposed in the next subsection for identifying the feasible solution to (\ref{equ_sys}).

		Next we equivalently transform the problem of identifying the existence of a feasible solution to (\ref{equ_sys})
		into an optimization problem.
		
		\subsection{Non-convex Optimization Problem}\label{bcd-alg}
		We first formulate condition \eqref{equ_sys} as a  non-convex optimization problem,  then use the ({\bf AM}) approach developed in  non-convex optimization theory to solve it. For the ({\bf AM}) approach, see \cite{bertsekas2015parallel}.
		
		To solve \eqref{equ_sys}, we reformulate it as an optimization problem:
		\begin{eqnarray}
			{\bf (OP):}\left\{
			\begin{array}{ll}
				& \min_{P,A} \Big\|PA-{\cal J}P \Big\|^2_F \label{july-15-1}\\
				& {\sf s.t.} \ \left\{
				\begin{array}{llll}
					& P {\bf 1}^\top={\bf 1}^\top;
					\\
					&\beta  P\geq {\bf 0};\\
					&A {\bf 1}^\top \leq {\bf 0};\\
					& -\xi\leq a_{ii}\leq 0 \ \mbox{for} \ 1\leq i\leq n;\\
					&0\leq a_{ij}\leq \xi \ \mbox{for} \ i\neq j, 1\leq i,j\leq n.
				\end{array}
				\right.
			\end{array}
			\right. \label{july-15-2}
		\end{eqnarray}
		Note that \eqref{equ_sys} is solvable if and only if the above optimization problem ({\bf OP}) has a solution such that the optimal objective value is zero. We use  the {\bf AM} approach to explore the optimal solution of the problem {\bf (OP)}.\\[0.15in]
		{\bf Step-a:} At the beginning, we pick up $A_0$ 
		and solve the following optimization problem  with $A=A_0$:
		\begin{eqnarray}
			\mbox{{\bf (OP[$A$])}}: \left \{
			\begin{array}{ll}
				&\min_{P} \Big\|P A-{\cal J}P \Big\|^2_F\\
				& {\sf s.t.} \ \left\{
				\begin{array}{ll}
					& P {\bf 1}^\top ={\bf 1}^\top;  \\
					&\beta P\geq {\bf 0}.
				\end{array}
				\right.
			\end{array}
			\right. \label{july-15-4}
		\end{eqnarray}
		
		\noindent
		{\bf Step-b:} Let $P_0$ denote the solution of the optimization problem {\bf (OP[$A$])} given by (\ref{july-15-4}) with $A=A_0$. Then we solve another optimization problem with $P=P_0$:
		\begin{eqnarray}
			\mbox{{\bf (OP[$P$])}}: \left \{
			\begin{array}{ll}
				&\min_{A} \Big\|P A-{\cal J} P \Big\|^2_F \label{july-15-5}\\
				& {\sf s.t.} \ \left\{
				\begin{array}{ll}
					&A {\bf 1}^\top \leq {\bf 0};\\
					& -\xi\leq a_{ii}\leq 0 \ \mbox{for} \ 1\leq i\leq n;\\
					&0\leq a_{ij}\leq \xi \ \mbox{for} \ i\neq j, 1\leq i,j\leq n.
				\end{array}
				\right.
			\end{array}
			\right. \label{july-15-6}
		\end{eqnarray}
		
		\noindent
		{\bf Step-c:} Let $A_1$ be the solution to (\ref{july-15-6}). Then solve the optimization problem {\bf (OP[$A$])} given by  (\ref{july-15-4}) with replacing  $A$ by $A_1$, and its solution is denoted by $P_1$. Again, we solve the optimization problem {\bf (OP[$P$])} given by  (\ref{july-15-6}) with replacing $P$ by $P_1$, and denote its solution by $A_2$. Repeating {\bf Step-a} and {\bf Step-b}, we generate a sequence $\{(A_l,P_l): l\geq 0\}$. \\[-0.05in]

		Take $P$ into the ${\cal J}$-blocked  partition, and rewrite $\Big\|PA-{\cal J}P \Big\|^2_F$  as
		\begin{align}\label{equ_ob}
			\begin{aligned}
				\Big\|PA-{\cal J}P \Big\|^2_F=&\sum_{i=1}^\ell\Big\|P(\lambda_i)A- {\cal J}_{n_i}(\lambda_i)P(\lambda_i) \Big\|^2_F\\
				&+\sum_{j=\ell+1}^k\Big\|P(\mu_{j},\omega_j)A- {\cal J}_{n_j}(\mu_{j},\omega_j)P(\mu_{j},\omega_j) \Big\|^2_F.
			\end{aligned}
		\end{align}
		To make the above steps workable, we establish the following two lemmas whose proofs are relegated to the Appendix.
		To get an expression for $\Big\|PA-{\cal J}P \Big\|^2_F$, let
		\begin{align}
			&\overline{P}(\lambda_i)=\Big(P_1(\lambda_i),P_2(\lambda_i),\ldots,P_{n_i}(\lambda_i)\Big);\label{pbar1}\\
			&\overline{P}(\mu_{j},\omega_j)=\Big(P^{(1)}_1(\mu_{j},\omega_j),P^{(1)}_2(\mu_{j},\omega_j),\ldots,P^{(n_j)}_1(\mu_{j},\omega_j),P^{(n_j)}_2
			(\mu_{j},\omega_j)\Big).\label{pbar2}
		\end{align}
		That is, $\overline{P}(\lambda_i)$ is an $(n\times n_i)$-dimensional row vector containing all rows  of $P(\lambda_i)$. Similarly, $\overline{P}(\mu_{j},\omega_j)$ is an $(n\times 2n_j)$-dimensional row vector containing all row vectors of $P(\mu_{j},\omega_j)$.
		The first lemma is about the convexity on the objective functions in {\bf Step-a} and {\bf Step-b}.
		\begin{lemma} \label{obj-convexity}
			For any given $A=(a_{ij})_{n\times n}\in {\cal R}^{n\times n}$ ,
			$\Big\|PA-{\cal J}P \Big\|^2_F$
			is convex with respect to $n^2$ variables $\{p_{ij}: 1\leq i,j\leq n\}$ given by the entries of matrix $P=(p_{ij})_{n\times n}$. Similarly, for any given $P=(p_{ij})_{n\times n}\in {\cal R}^{n\times n}$, $\Big\|PA-{\cal J}P \Big\|^2_F$
			is convex with respect to $n^2$ variables $\{a_{ij}: 1\leq i,j\leq n\}$ given by the entries of matrix $A=(a_{ij})_{n\times n}$.
			Moreover,
			\begin{align}
				&\Big\|P(\lambda_i)A- {\cal J}_{n_i}(\lambda_i)P(\lambda_i) \Big\|^2_F=\overline{P}(\lambda_i)B_{\lambda_i}B_{\lambda_i}^\top\overline{P}(\lambda_i)^\top;\label{conv_pr}\\
				&\Big\|P(\mu_{j},\omega_j)A- {\cal J}_{n_j}(\mu_{j},\omega_j)P(\mu_{j},\omega_j) \Big\|^2_F=\overline{P}(\mu_{j},\omega_j)B_{(\mu_{j},\omega_j)}B_{(\mu_{j},\omega_j)}^\top\overline{P}(\mu_{j},\omega_j)^\top,\label{conv_pc}
			\end{align}
			where $B_{\lambda_i}$ is an $((nn_i)\times (nn_i))$-matrix
			\begin{align}\label{ob_exp}
				B_{\lambda_i}=
				\begin{bmatrix}
					A-\lambda_iI&-I & &\\
					& A-\lambda_iI &\ddots &\\
					& & \ddots & -I\\
					& & &A-\lambda_iI
				\end{bmatrix}
			\end{align}
			and $B_{(\mu_{j},\omega_j)}$ is a $((2nn_j)\times (2nn_j))$-matrix
			\begin{align}
				\begin{bmatrix}
					A_{(\mu_j,\omega_j)}&-I & &\\
					& A_{(\mu_j,\omega_j)} &\ddots &\\
					& & \ddots & -I\\
					& & &A_{(\mu_j,\omega_j)}
				\end{bmatrix}
				\mbox{with} \ \
				A_{(\mu_j,\omega_j)}=
				\begin{bmatrix}
					A-\mu_jI&-\omega_jI\\
					\omega_j I& A-\mu_jI
				\end{bmatrix}.\label{B-defin}
			\end{align}	
		\end{lemma}
		
		The next lemma shows us that  both the convex optimization problem {\bf (OP[$A$])} given by {\rm (\ref{july-15-4})} and the convex optimization problem {\bf (OP[$P$])} given by {\rm (\ref{july-15-6})} are well posed.
		
		\begin{lemma}\label{Alter-optimal-solution}
			For any given $A\in {\cal R}^{n\times n}$, the convex optimization problem {\bf (OP[$A$])} given by {\rm (\ref{july-15-4})} has an optimal solution if the set of the feasible solutions to its constraints is nonempty$;$ For any given $P\in {\cal R}^{n\times n}$, the convex optimization problem {\bf (OP[$P$])} given by {\rm (\ref{july-15-6})} has an optimal solution.
		\end{lemma}	
		
		Note that if $\beta$ given by Proposition \ref{prop_beta} leads to no feasible solution to the convex optimization problem {\bf (OP[$A$])} given by (\ref{july-15-4}), then the (in)equalities \eqref{equ_sys} for such $\beta$ in Theorem \ref{iff} has no feasible solution. This gives that the associated rational function $ {\cal L}(s)$ given by \eqref{prob1} has no $n$-order continuous-time PH representation.
		
		To verify the {\bf AM} approach workable to the non-convex optimization problem {\bf (OP)}, following the procedure of the {\bf AM}, we need to establish the convergence of $\{(A_l,P_l): l\geq 0\}$ and that the corresponding limit is a  {\it critical point}. Here a critical point for problem $\min\{g(x),x\in X\}$ is a point $x\in X$ such that
		$\bigtriangledown g(x)^\top(y-x)\geq 0$ for each $y\in X$, where $\bigtriangledown g(x)$ denotes the gradient of $g$ at $x$.

		To this end, we first establish the boundedness of the optimal solution to {\bf (OP[$A$])} given by (\ref{july-15-4}) for any given $A \in {\cal R}^{n\times n}$ satisfying the constraints in \eqref{july-15-6}.
		\begin{proposition}\label{prop-ubound}
			For any given $A \in {\cal R}^{n\times n}$ satisfying the constraints in \eqref{july-15-6}, the optimal solution to {\bf (OP[$A$])} given by {\rm (\ref{july-15-4})}  is uniformly bounded.
		\end{proposition}
		
		\begin{proof}
			Denote by {\sf Sol}{\bf (OP[$A$])} the solution set of {\bf (OP[$A$])}. We first show that fixed $A$, {\sf Sol}{\bf (OP[$A$])} is bounded.
			From lemma \ref{obj-convexity},	{\bf (OP[$A$])} is a quadratic convex optimization. So its optimal solution set is the same as its local optimal solution set.
			If {\sf Sol}{\bf (OP[$A$])} is unbounded, by Theorem 4.2 in Lee et al. \cite{lee2005quadratic}, there exist $P\in {\cal R}^{n\times n}$ and  $Q \in {\cal R}^{n\times n}$ with $Q\neq {\bf 0}$ such that $\{P+t Q, t\geq 0\} \in$ {\sf Sol}{\bf (OP[$A$])}.
			Then $P$ and $Q$ must satisfy
			\begin{align}\label{sol_1}
				\begin{aligned}
					&P {\bf 1}^\top={\bf 1}^\top, \quad (P+t Q){\bf 1}^\top={\bf 1}^\top,\\
					&\beta P\geq {\bf 0}, \quad \beta(P+t Q)\geq {\bf 0},
				\end{aligned}
			\end{align}
			and
			\begin{align}\label{sol_2}
				\|PA-{\cal J}P\|^2_F = \|(P+t Q)A-{\cal J}(P+t Q)\|^2_F
			\end{align}
			for $t\geq 0$. Since $t$ is arbitrary nonnegative real number,
			\eqref{sol_1} implies that $Q{\bf 1}^\top={\bf 0}$ and $\beta Q\geq {\bf 0}$. Thus we have $\beta Q{\bf 1}^\top=0$. This together with  $\beta Q\geq {\bf 0}$ gives us $\beta Q={\bf 0}$.
			We make ${\cal J}$-blocked partition for $Q$, and similar to (\ref{pbar1})-(\ref{pbar2}), let
			\begin{align*}
				&\overline{Q}(\lambda_i)=\Big(Q_1(\lambda_i),Q_2(\lambda_i),\ldots,Q_{n_i}(\lambda_i)\Big);\\	&\overline{Q}(\mu_{j},\omega_j)=\Big(Q^{(1)}_1(\mu_{j},\omega_j),Q^{(1)}_2(\mu_{j},\omega_j),\ldots,Q^{(n_j)}_1(\mu_{j},\omega_j),
				Q^{(n_j)}_2(\mu_{j},\omega_j)\Big).
			\end{align*}
			By \eqref{equ_ob}-\eqref{conv_pc}, \eqref{sol_2} gives that for $t\geq 0$,
			\begin{align*}	&\sum_{i=1}^\ell\overline{P}(\lambda_i)B_{\lambda_i}B_{\lambda_i}^\top\overline{P}(\lambda_i)^\top+\sum_{j=\ell+1}^k\overline{P}(\mu_{j},\omega_j)
				B_{(\mu_{j},\omega_j)}B_{(\mu_{j},\omega_j)}^\top\overline{P}(\mu_{j},\omega_j)^\top\\ &=\sum_{i=1}^\ell\Big(\overline{P}(\lambda_i)+t\overline{Q}(\lambda_i)\Big)B_{\lambda_i}B_{\lambda_i}^\top
				\Big(\overline{P}(\lambda_i)+t\overline{Q}(\lambda_i)\Big)^\top\\
				&\quad +\sum_{j=\ell+1}^k\Big(\overline{P}(\mu_{j},\omega_j)+t\overline{Q}(\mu_{j},\omega_j)\Big)B_{(\mu_{j},\omega_j)}B_{(\mu_{j},\omega_j)}^\top
				\Big(\overline{P}(\mu_{j},\omega_j)+t\overline{Q}(\mu_{j},\omega_j)\Big)^\top.
			\end{align*}
			It can be directly calculated that
			\begin{align*} &t^2\left(\sum_{i=1}^\ell\overline{Q}(\lambda_i)B_{\lambda_i}B_{\lambda_i}^\top\overline{Q}(\lambda_i)^\top
				+\sum_{j=\ell+1}^k\overline{Q}(\mu_{j},\omega_j)B_{(\mu_{j},\omega_j)}B_{(\mu_{j},\omega_j)}^\top\overline{Q}(\mu_{j},\omega_j)^\top\right)\\ &+2t\left(\sum_{i=1}^\ell\overline{P}(\lambda_i)B_{\lambda_i}B_{\lambda_i}^\top\overline{Q}(\lambda_i)^\top
				+\sum_{j=\ell+1}^k\overline{P}(\mu_{j},\omega_j)B_{(\mu_{j},\omega_j)}B_{(\mu_{j},\omega_j)}^\top\overline{Q}(\mu_{j},\omega_j)^\top\right)=0
			\end{align*}
			for $t\geq 0$. Since $t$ is arbitrary nonnegative real number, we have
			\begin{align*}
				\sum_{i=1}^\ell\overline{Q}(\lambda_i)B_{\lambda_i}B_{\lambda_i}^\top\overline{Q}(\lambda_i)^\top
				+\sum_{j=\ell+1}^k\overline{Q}(\mu_{j},\omega_j)B_{(\mu_{j},\omega_j)}B_{(\mu_{j},\omega_j)}^\top\overline{Q}(\mu_{j},\omega_j)^\top=0.
			\end{align*}
			Note that each term is nonnegative. Thus for $1\leq i\leq \ell$ and $\ell+1\leq j\leq k$,
			\begin{align*}
				&\overline{Q}(\lambda_i)B_{\lambda_i}B_{\lambda_i}^\top\overline{Q}(\lambda_i)^\top=0,\\
				&\overline{Q}(\mu_{j},\omega_j)B_{(\mu_{j},\omega_j)}B_{(\mu_{j},\omega_j)}^\top\overline{Q}(\mu_{j},\omega_j)^\top=0.
			\end{align*}
			Equivalently,
			\begin{align}
				&\overline{Q}(\lambda_i)B_{\lambda_i}={\bf 0},\label{equ:qr}\\
				&\overline{Q}(\mu_{j},\omega_j)B_{(\mu_{j},\omega_j)}={\bf 0}.\label{equ:qc}
			\end{align}
			We first focus on equation \eqref{equ:qr}, by the definition of $B_{\lambda_i}$ in (\ref{ob_exp}),
			\begin{align*}
				\overline{Q}(\lambda_i)B_{\lambda_i}=\Big(Q_1(\lambda_i),Q_2(\lambda_i),\ldots,Q_{n_i}(\lambda_i)\Big)\cdot
				\begin{bmatrix}
					A-\lambda_iI&-I & &\\
					& A-\lambda_iI &\ddots &\\
					& & \ddots & -I\\
					& & &A-\lambda_iI
				\end{bmatrix}= {\bf 0}.
			\end{align*}
			This gives that
			\begin{align*}
				&Q_1(\lambda_i)\cdot(A-\lambda_iI)={\bf 0},\\
				&Q_r(\lambda_i)\cdot(A-\lambda_iI)=Q_{r-1}(\lambda_i), \quad 2\leq r\leq n_i.
			\end{align*}
			Note that if there exists $r$ such that $Q_r(\lambda_i)={\bf 0}$, then $Q_l(\lambda_i)={\bf 0}$ for all $l\leq r$.
			If $Q_{n_i}(\lambda_i)\neq {\bf 0}$, let
			$
			\sigma_i=\min\{r:Q_r(\lambda_i)\neq {\bf 0}\}.
			$ 
			Otherwise, let $\sigma_i=n_i+1$.
			Hence, when $\sigma_i\leq n_i$, we have $Q_l(\lambda_i)={\bf 0}$ for all $l< \sigma_i$ and $Q_l(\lambda_i)\neq{\bf 0}$ for all $\sigma_i\leq l\leq n_i$.
			
			For equation \eqref{equ:qc}, by the definition of $B_{(\mu_j,\omega_j)}$ in (\ref{B-defin}),
			\begin{align*}
				\overline{Q}(\mu_{j},\omega_j)B_{(\mu_{j},\omega_j)}
				&=\Big(Q^{(1)}_1(\mu_{j},\omega_j),Q^{(1)}_2(\mu_{j},\omega_j),\ldots, Q^{(n_j)}_1(\mu_{j},\omega_j)_{(1)},Q^{(n_j)}_2(\mu_{j},\omega_j)_{(2)}\Big)\\
				& \ \ \ \ \times
				\begin{bmatrix}
					A_{(\mu_j,\omega_j)}&-I & &\\
					& A_{(\mu_j,\omega_j)} &\ddots &\\
					& & \ddots & -I\\
					& & &A_{(\mu_j,\omega_j)}
				\end{bmatrix}= {\bf 0}.
			\end{align*}
			Equivalently,	
			\begin{align*}
				&Q^{(1)}_1(\mu_{j},\omega_j)\cdot(A-\mu_jI)+\omega_jQ^{(1)}_2(\mu_{j},\omega_j)={\bf 0},\\
				&Q^{(1)}_2(\mu_{j},\omega_j)\cdot(A-\mu_jI)-\omega_jQ^{(1)}_1(\mu_{j},\omega_j)={\bf 0},\\
				&Q^{(r)}_1(\mu_{j},\omega_j)\cdot(A-\mu_jI)+\omega_jQ^{(r)}_2(\mu_{j},\omega_j)
				=Q^{(r-1)}_1(\mu_{j},\omega_j),\quad 2\leq r\leq n_j,\\
				&Q^{(r)}_2(\mu_{j},\omega_j)\cdot(A-\mu_jI)-\omega_jQ^{(r)}_1(\mu_{j},\omega_j)=Q^{(r-1)}_2(\mu_{j},\omega_j),\quad 2\leq r\leq n_j.
			\end{align*}	
			Letting the first equation plus the second equation multiplying $\imath$, we obtain the following first equation; Letting the first equation minus the second equation multiplying $\imath$, we obtain the following second equation. Similarly, for $2\leq r\leq n_j$, letting the third equation plus the last equation multiplying $\imath$, we have the following third equation; Letting the third equation minus the last equation multiplying $\imath$, we have the following last equation:
			\begin{align}\label{equ:c_mul}
				\begin{aligned}
					&\Big(Q^{(1)}_1(\mu_{j},\omega_j)+\imath Q^{(1)}_2(\mu_{j},\omega_j)\Big)\cdot(A-\mu_jI-\imath \omega_jI)={\bf 0},\\
					&\Big(Q^{(1)}_1(\mu_{j},\omega_j)-\imath Q^{(1)}_2(\mu_{j},\omega_j)\Big)\cdot(A-\mu_jI+\imath \omega_jI)={\bf 0},\\
					&\Big(Q^{(r)}_1(\mu_{j},\omega_j)+\imath Q^{(r)}_2(\mu_{j},\omega_j)\Big)\cdot(A-\mu_jI-\imath \omega_jI)\\
					&\quad\quad=Q^{(r-1)}_1(\mu_{j},\omega_j)+\imath Q^{(r-1)}_2(\mu_{j},\omega_j), \ 2\leq r\leq n_j,\\
					&\Big(Q^{(r)}_1(\mu_{j},\omega_j)-\imath Q^{(r)}_2(\mu_{j},\omega_j)\Big)\cdot(A-\mu_jI+\imath \omega_jI)\\
					&\quad\quad=Q^{(r-1)}_1(\mu_{j},\omega_j)-\imath Q^{(r-1)}_2(\mu_{j},\omega_j),\ 2\leq r\leq n_j.
				\end{aligned}
			\end{align}	
			If $Q^{(n_j)}_l(\mu_{j},\omega_j)\neq {\bf 0}$ for $l=1,2$, let
			$$
			\zeta_{j}=\min\{r:Q^{(r)}_1(\mu_{j},\omega_j)\neq {\bf 0}\},\  \
			\xi_j=\min\{r:Q^{(r)}_2(\mu_{j},\omega_j)\neq {\bf 0}\}. $$
			Otherwise, 	if $Q^{(n_j)}_1(\mu_{j},\omega_j)= {\bf 0}$, let $\zeta_{j}=n_j+1$, and if $Q^{(n_j)}_2(\mu_{j},\omega_j)= {\bf 0}$, let $\xi_{j}=n_j+1$.
			We first claim that $\zeta_{j}=\xi_j$. If it is not true, assume that $\zeta_{j}>\xi_j$. Using (\ref{equ:c_mul}) with $r=\xi_j$, and by
			$Q^{(\xi_j)}_1(\mu_{j},\omega_j)=0$ as $\zeta_{j}>\xi_j$,
			\begin{align*}
				\imath Q^{(\xi_j)}_2(\mu_{j},\omega_j)\cdot(A-\mu_jI-\imath \omega_jI)={\bf 0},\ \
				-\imath Q^{(\xi_j)}_2(\mu_{j},\omega_j)\cdot(A-\mu_jI+\imath \omega_jI)={\bf 0}.
			\end{align*}
			Add up these two equations, $\omega_j\neq 0$ implies that $Q^{(\xi_j)}_2(\mu_{j},\omega_j)={\bf 0}.$ It contradicts the definition of $\xi_j$. Conversely, if $\zeta_{j}<\xi_j$, the proof is similar.	
			
			Now we show that $Q^{(r)}_1(\mu_{j},\omega_j)=Q^{(r)}_2(\mu_{j},\omega_j)={\bf 0}$ for $r<\zeta_j$, and $Q^{(r)}_1(\mu_{j},\omega_j)\neq{\bf 0}$, $Q^{(r)}_2(\mu_{j},\omega_j)\neq{\bf 0}$ for $\zeta_j\leq r\leq n_j$.
			The former is obvious by the definition of $\xi_j$. We only show the latter. If there exists $\zeta_j< r\leq n_j$ such that $Q^{(r)}_1(\mu_{j},\omega_j)={\bf 0}$ or $Q^{(r)}_2(\mu_{j},\omega_j)={\bf 0}$, let
			\begin{align*}
				\zeta_{j}^0=\min\{l>\zeta_j:Q^{(l)}_1(\mu_{j},\omega_j)= {\bf 0}\},\ \
				\xi_j^0=\min\{l>\zeta_j:Q^{(l)}_2(\mu_{j},\omega_j)= {\bf 0}\}.
			\end{align*} 
			We first show that $\zeta_{j}^0=\xi_j^0$. Otherwise, if $\zeta_{j}^0>\xi_j^0$, from \eqref{equ:c_mul},
			\begin{align*}
				&Q^{(\xi_j^0)}_1(\mu_{j},\omega_j)\cdot(A-\mu_jI-\imath \omega_jI)=Q^{(\xi_j^0-1)}_1(\mu_{j},\omega_j)+\imath Q^{(\xi_j^0-1)}_2(\mu_{j},\omega_j),\\
				&Q^{(\xi_j^0)}_1(\mu_{j},\omega_j)\cdot(A-\mu_jI+\imath \omega_jI)=Q^{(\xi_j^0-1)}_1(\mu_{j},\omega_j)-\imath Q^{(\xi_j^0-1)}_2(\mu_{j},\omega_j),\\
				&\Big(Q^{(\xi_j^0-1)}_1(\mu_{j},\omega_j)+\imath Q^{(\xi_j^0-1)}_2(\mu_{j},\omega_j)\Big)\cdot(A-\mu_jI-\imath \omega_jI)^{\xi_j^0-1-\zeta_j}\\
				&\quad\quad=Q^{(\zeta_j)}_1(\mu_{j},\omega_j)+\imath Q^{(\zeta_j)}_2(\mu_{j},\omega_j),\\
				&\Big(Q^{(\xi_j^0-1)}_1(\mu_{j},\omega_j)-\imath Q^{(\xi_j^0-1)}_2(\mu_{j},\omega_j)\Big)\cdot(A-\mu_jI+\imath \omega_jI)^{\xi_j^0-1-\zeta_j}\\
				&\quad\quad=Q^{(\zeta_j)}_1(\mu_{j},\omega_j)-\imath Q^{(\zeta_j)}_2(\mu_{j},\omega_j),\\
				&\Big(Q^{(\zeta_j)}_1(\mu_{j},\omega_j)+\imath Q^{(\zeta_j)}_2(\mu_{j},\omega_j)\Big)\cdot(A-\mu_jI-\imath \omega_jI)={\bf 0},\\
				&\Big(Q^{(\zeta_j)}_1(\mu_{j},\omega_j)-\imath Q^{(\zeta_j)}_2(\mu_{j},\omega_j)\Big)\cdot(A-\mu_jI+\imath \omega_jI)={\bf 0}.
			\end{align*}
			Note that $Q^{(l)}_1(\mu_{j},\omega_j)\neq {\bf 0}$ and $Q^{(l)}_2(\mu_{j},\omega_j)\neq {\bf 0}$ for $\zeta_j\leq l\leq \xi_j^0-1$. This implies that $Q^{(\xi_j^0)}_1(\mu_{j},\omega_j)$ is the generalized eigenvector
			corresponding to both $\mu_j+\imath \omega_j$ and $\mu_j-\imath \omega_j$. It leads to a contradiction. Conversely, if $\zeta_{j}^0<\xi_j^0$, the proof is similar.
			
			Now we have $\zeta_{j}^0=\xi_j^0$. This together with \eqref{equ:c_mul} gives that $$Q^{(r)}_1(\mu_{j},\omega_j)=Q^{(r)}_2(\mu_{j},\omega_j)={\bf 0},\ r<\zeta_j^0.$$  It leads to a contradiction.
			Then we obtain that $Q^{(r)}_1(\mu_{j},\omega_j)\neq{\bf 0}$, $Q^{(r)}_2(\mu_{j},\omega_j)\neq{\bf 0}$ for $\zeta_j\leq r\leq n_j$.
			
			Let
			${\cal N}=\{1\leq i \leq \ell:\sigma_i\leq n_i\}\cup\{\ell+1\leq j\leq k: \zeta_j\leq n_j\}.$ 
			It follows from $Q\neq {\bf 0}$ that ${\cal N}$ is not empty. In summary, for  $i,j\in {\cal N}$,
			\begin{align*}
				&Q_{\sigma_i}(\lambda_i)\cdot(A-\lambda_iI)={\bf 0},\\
				&Q_r(\lambda_i)\cdot(A-\lambda_iI)=Q_{r-1}(\lambda_i), \quad \sigma_i+1\leq r\leq n_i
			\end{align*}
			and
			\begin{align*}
				&\Big(Q^{(\zeta_j)}_1(\mu_{j},\omega_j)+\imath Q^{(\zeta_j)}_2(\mu_{j},\omega_j)\Big)\cdot(A-\mu_jI-\imath \omega_jI)={\bf 0},\\
				&\Big(Q^{(\zeta_j)}_1(\mu_{j},\omega_j)-\imath Q^{(\zeta_j)}_2(\mu_{j},\omega_j)\Big)\cdot(A-\mu_jI+\imath \omega_jI)={\bf 0},\\
				&\Big(Q^{(r)}_1(\mu_{j},\omega_j)+\imath Q^{(r)}_2(\mu_{j},\omega_j)\Big)\cdot(A-\mu_jI-\imath \omega_jI)\\
				&\quad \quad=Q^{(r-1)}_1(\mu_{j},\omega_j)+\imath Q^{(r-1)}_2(\mu_{j},\omega_j),\ \zeta_j+1\leq r\leq n_j,\\
				&\Big(Q^{(r)}_1(\mu_{j},\omega_j)-\imath Q^{(r)}_2(\mu_{j},\omega_j)\Big)\cdot(A-\mu_jI+\imath \omega_jI)\\
				&\quad \quad=Q^{(r-1)}_1(\mu_{j},\omega_j) -\imath Q^{(r-1)}_2(\mu_{j},\omega_j) , \ \zeta_j+1\leq r\leq n_j.
			\end{align*}	
			It follows that for matrix $A$, $Q_{\sigma_i}(\lambda_i)$ is the eigenvector related to $\lambda_i$ and $Q_r(\lambda_i)$ is the generalized eigenvector related to $\lambda_i$ for $\sigma_i+1\leq r\leq n_i$; $Q^{(\zeta_j)}_1(\mu_{j},\omega_j)\pm \imath Q^{(\zeta_j)}_2(\mu_{j},\omega_j)$ are the eigenvectors related to $\mu_j\pm \imath\omega_j$ and $Q^{(r)}_1(\mu_{j},\omega_j)\pm \imath Q^{(r)}_2(\mu_{j},\omega_j)$	are the generalized eigenvectors related to $\mu_j\pm \imath \omega_j$ for $\zeta_j+1\leq r\leq n_j$. Thus, they are linearly independent.
			This contradicts $\beta Q={\bf 0}$ from Proposition \ref{prop_beta}. Therefore, such $Q$ does not exist. This implies that the optimal solution  set to {\bf (OP[$A$])} is bounded.

			Now we show that {\sf Sol}{\bf (OP[$A$])} is uniformly bounded for $A$ satisfying constraints in \eqref{july-15-6}.
			If it is not true, there must exist $\{A_l: \ l\geq 1\}$ such that
			\begin{align}\label{p_inf}
				\sup_{P_l\in {\sf Sol} \mbox{{\bf (OP[$A_l$])} }} \big \|P_l\big \|_F^2 \rightarrow \infty.
			\end{align}
			Note that $\{A_l: \ l\geq 1\}$ belongs to a compact set. Without loss of generality, we assume that $\{A_l: \ l\geq 1\}$ converges to $A_\infty$. By equation \eqref{equ_ob}-\eqref{conv_pc}, for each {\bf (OP[$A_l$])} with $1\leq l\leq \infty$, its objective function is quadratic and the corresponding Hessian matrix is positive semi-definite. Consider the following affine variational
			inequality problem denoted by {\bf (AVI[$A_l$])}:
			\begin{eqnarray*}
				\left\{
				\begin{array}{lll}
					&\mbox{Find $P\in {\cal R}^{n\times n}$ such that}\\
					&\mbox{for all $U\in {\cal R}^{n\times n}$ with $U{\bf 1}^{\top}={\bf 1}^{\top}$ and $\beta U\geq {\bf 0}$, we have}\\
					&\sum_{i=1}^\ell\Big\langle B_{\lambda_i}B^\top_{\lambda_i}\overline{P}^\top(\lambda_i), \overline{U}^\top(\lambda_i)-\overline{P}^\top(\lambda_i)\Big\rangle\\
					&+\sum_{j=\ell+1}^k\Big\langle B_{(\mu_{j},\omega_j)}B^\top_{(\mu_{j},\omega_j)}\overline{P}^\top(\mu_{j},\omega_j),\overline{U}^\top(\mu_{j},\omega_j)-
					\overline{P}^\top(\mu_{j},\omega_j)\Big\rangle\geq 0,\\
					&\mbox{where $\overline{P}(\lambda_i)$, $\overline{U}(\lambda_i)$, $\overline{P}(\mu_{j},\omega_j)$ and $\overline{U}(\mu_{j},\omega_j)$ are given by the ${\cal J}$-blocked}\\
					&\mbox{ partitions of $P$ and $U$ as (\ref{pbar1})-(\ref{pbar2}), and $B_{\lambda_i}$ and $B_{(\mu_j\omega_j)}$ are given by}\\
					&\mbox{(\ref{ob_exp})-(\ref{B-defin}) with replacing $A$ by $A_l$.}
				\end{array}
				\right.
			\end{eqnarray*}
			Using Corollary 5.2 in Lee et al. \cite{lee2005quadratic}, the solution set of {\bf (AVI[$A_l$])} is the same as the KKT solution set of problem {\bf (OP[$A_l$])}, and then it is the same as the optimal solution set of problem {\bf (OP[$A_l$])} as it is convex. However, by Theorem 7.6 in Lee et al. \cite{lee2005quadratic},  we know the solution sets of {\bf (AVI[$A_l$])} for $l\geq 1$ are uniformly bounded.
			Therefore, the optimal solution sets for {\sf Sol}{\bf (OP[$A_l$])} for $l\geq 1$ are uniformly bounded.
			This implies that (\ref{p_inf}) is not true. Thus, we complete the proof of this proposition.	
		\end{proof}

		Now we can get the effectiveness of the {\bf AM} for the non-convex optimization problem {\bf (OP)} given by (\ref{july-15-2}). For simplicity, let ${\cal F} (A,P)=\Big\|PA-{\cal J}P \Big\|^2_F$.
		\begin{theorem}\label{KKT}
			The algorithm $($given by {\bf Step-a} and {\bf Step-b}$)$ can guarantee$:$
			\begin{itemize}
				\item[{\rm (i)}]  The sequence $\{{\cal F}(A_l,P_l):l\geq 0\}$  has a finite limit denoted by ${\cal F}_\infty;$
				\item[{\rm (ii)}] The sequence $\{(A_l, P_l): l\geq 0\}$ admits at least one limit point, and its any limit point, say $( {A}_\infty,{P}_\infty)$, is a critical point of the non-convex optimization problem $(\bf OP)$ given by {\rm (\ref{july-15-2})}. Further, ${\cal F}( {A}_\infty,{P}_\infty)= {\cal F}_\infty$.
			\end{itemize}
		\end{theorem}
		
		\begin{proof}
			We first prove that $\{{\cal F}(A_l, P_l): l\geq 0\}$ is monotonically decreasing.
			Recalling the procedure to generate $\{(A_l, P_l): l\geq 0\}$ by {\bf Step-a} and {\bf Step-b}, we first arbitrarily
			pick up $A_0$ which satisfies the constraints in \eqref{july-15-6}, and
			\begin{align*}
				\left\{
				\begin{array}{llll}
					&P_l \ \mbox{is a solution for the convex optimization problem (\bf OP[$A$])}\\
					& \ \mbox{given by (\ref{july-15-4}) with $A=A_{l}$ for } l\geq 0;\\
					& A_l \ \mbox{is a solution for the convex optimization problem (\bf{OP[$P$]})}\\
					& \ \mbox{given by (\ref{july-15-6}) with $P=P_{l-1}$ for }   l\geq 1.
				\end{array}
				\right.
			\end{align*}	
			Hence we have that for $l\geq 0$,
			\begin{align*}
				{\cal F}(A_l, P_l)\geq {\cal F}(A_{l+1}, P_l) \geq {\cal F}(A_{l+1}, P_{l+1}).
			\end{align*}	
			Since ${\cal F}(A_l, P_l) \geq 0$ for each $l$, there exists a nonnegative real number $ {\cal F}_\infty$ such that
			\begin{eqnarray*}
				\lim_{l \rightarrow \infty} {\cal F}(A_l, P_l) ={\cal F}_\infty.
			\end{eqnarray*}
			This gives the first part of the theorem.

			Now prove the second part of the theorem. The existence of the limit point admitted by the sequence $\{(A_l, P_l): l\geq 0\}$ directly follows from
			Proposition \ref{prop-ubound}.
			Further, by the first part of the theorem, we know that the objective value given by any limit point of
			$\{(A_l, P_l): l\geq 0\}$ must be ${\cal F}_\infty$. So the remainder of the proof of the second part is to show that
			any limit point of  $\{(A_l, P_l): l\geq 0\}$  must be a critical point of the non-convex optimization problem $(\bf OP)$ given by {\rm (\ref{july-15-2})}. From
			Proposition \ref{prop-ubound}, this follows from Corollary 2 in \cite{grippo2000convergence}. Then we complete this theorem.
		\end{proof}
		
		Telek and Horv\'ath \cite{telek2007} develop a moments matching algorithm to determine an $n$-order PH representation of the rational function ${\cal L}(s)$ given by \eqref{prob1}. To make their algorithm possibly converge to a PH distribution,
		continuously differentiable functions used to penalize the negative off-diagonal terms must be properly constructed and implemented in each step of the algorithm.  When their algorithm gets convergence by keeping properly tuning the penalization functions, if it terminates with a PH representation, then it successfully finds an $n$-order PH representation of ${\cal L}(s)$; Otherwise, if it terminates without a PH representation, it cannot tell us if ${\cal L}(s)$ has an $n$-order PH representation. Here, Theorem \ref{KKT}, however, shows that the {\bf AM} algorithm converges, and the convergence limit point is a critical point without keeping properly constructing the penalization functions. Similar to the moments matching algorithm of Telek and Horv\'ath \cite{telek2007}, on the other hand,  as the outcome objective value obtained from the {\bf AM} algorithm may not be the global minimum, when the outcome objective value computed by the {\bf AM} algorithm turns out to be non-zero, it does not imply that ${\cal L}(s)$ has no $n$-order PH representation, see Example \ref{examp:dis} in Section \ref{sec:numerical}.
		
		Next we show that the above analysis and results for the continuous-time minimal PH representation
		problem can be used to analyze the discrete-time case.
		
		\section{ Minimal Discrete-time PH Representation}\label{dct_case}
		Parallel to the continuous-time case, consider a discrete nonnegtaive-integer-value random variable with the generating function
		\begin{align}\label{prob2}
			{\cal G}(z)=\frac{\widetilde{p}(z)}{\widetilde{q}(z)}=\frac{\widetilde{p}_nz^n+\widetilde{p}_{n-1}z^{n-1}+\cdots +\widetilde{p}_1z}{\widetilde{q}_nz^n+\widetilde{q}_{n-1}z^{n-1}+\cdots+1} \ \ \mbox{for  $|z|\leq 1$}.
		\end{align}
		If ${\cal G}(z)$ satisfies that
		\begin{itemize}
			\item [$(\widetilde{\bf A1})$] its coefficients are real numbers;
			\item [$(\widetilde{\bf A2})$] $\widetilde{p}_n^2+\widetilde{q}_n^2\neq 0$, $\widetilde{p}(s)$ and $\widetilde{q}(s)$ are coprime;
			\item [$(\widetilde{\bf A3})$] ${\cal G}(1)=1$ and the root with minimal modulus of $\widetilde{q}(z)$ is real and
			strictly greater than $1$,
		\end{itemize}
		then, from Theorem 1.2 in O'Cinneide \cite{o1990characterization}, we know that this discrete nonnegtaive-integer-value random variable
		is given by an absorbing time of a discrete-time Markov chain. Similar to the continuous-time case, starting with the moment generating function ${\cal G}(z)$ given by (\ref{prob2}) with $(\widetilde{\bf A1})$-$(\widetilde{\bf A3})$ holding,
		our problem can be phrased as:
		\begin{itemize}
			\item[] {\it Finding  $\widetilde{\alpha} \in {\cal R}_+^n$ with $\widetilde\alpha{\bf 1}^\top=1$, and
				a substochastic matrix $\widetilde{A}=(\widetilde{a}_{ij})_{n\times n}\in {\cal R}_+^{n\times n}$
				with $(I-\widetilde{A})$ being nonsingular
				such that
				\begin{align}
					{\cal G}(z)&=\sum_{l=1}^{\infty}\widetilde\alpha \widetilde{A}^{l-1}(I-\widetilde{A}){\bf 1}^\top z^l\nonumber\\
					&=z\widetilde\alpha (I-\widetilde{A})(I-z\widetilde{A})^{-1}{\bf 1}^\top\nonumber\\
					&=\frac{z\widetilde\alpha (I-\widetilde{A})\cdot {\sf adj}(I-z\widetilde{A}){\bf 1}^\top}{{\sf det}(I-z\widetilde{A})} \ \ \mbox{for $|z|\leq 1$}.	\label{equ:dis-ge}		
			\end{align}}
		\end{itemize}

		Note that the form of \eqref{prob2} with $(\widetilde{\bf A1})$-$(\widetilde{\bf A3})$ holding actually includes the generating functions of all discrete-time PH distributions which assign no mass at zero and have no periodicity. Here a discrete-time PH distribution has no periodicity means that it has a representation whose corresponding Markov chain has no periodic states. To see this, suppose that $(\widetilde \alpha,\widetilde{A})$ is such a representation from an aperiodic, and $(m+1)$-state Markov chain in which (i) all the states are transient except one state is absorbing; (ii) the initial state distribution given by $\widetilde\alpha$ has zero mass at the absorbing state; and (iii) the transition probabilities among the transient states is $\widetilde{A}$. We claim that $z\widetilde\alpha (I-\widetilde{A})(I-z\widetilde{A})^{-1}{\bf 1}^\top$ can be written as the form of \eqref{prob2} and satisfies $(\widetilde{\bf A1})$-$(\widetilde{\bf A3})$. Notice first that the numerator polynomial in $z\widetilde\alpha (I-\widetilde{A})(I-z\widetilde{A})^{-1}{\bf 1}^\top$  has a factor $z$ but  the denominator polynomial does not have it. Therefore, after some simplification, when the numerator polynomial and denominator polynomial become coprime, the numerator polynomial has no constant term and then
		$z\widetilde\alpha (I-\widetilde{A})(I-z\widetilde{A})^{-1}{\bf 1}^\top$ is the form of \eqref{prob2}.
		Further note that
		\[
		\Big(z\widetilde\alpha (I-\widetilde{A})(I-z\widetilde{A})^{-1}{\bf 1}^\top\Big)_{z=1}=1.
		\]
		Hence, if $z\widetilde\alpha (I-\widetilde{A})(I-z\widetilde{A})^{-1}{\bf 1}^\top$ is a polynomial, then it satisfies $(\widetilde{\bf A1})$-$(\widetilde{\bf A3})$.
		Otherwise, it is straightforward to see that $z\widetilde\alpha (I-\widetilde{A})(I-z\widetilde{A})^{-1}{\bf 1}^\top$
		satisfies $(\widetilde{\bf A1})$-$(\widetilde{\bf A2})$. To see $(\widetilde{\bf A3})$, it suffices to show that its pole with minimal modulus is
		real and strictly larger than 1. From the proof of the necessity of Theorem 1.2 on p.50 in O'Cinneide \cite{o1990characterization}, its pole with minimal modulus is real and positive. To see it is strictly larger than 1, note that its poles belong to the set of reciprocals of $\widetilde{A}$'s nonzero eigenvalues. As $\widetilde{A}$ is substochastic and $(I-\widetilde{A})$ is invertible, the spectral radius of $\widetilde{A}$ is strictly less than 1. Hence, the pole of $z\widetilde\alpha (I-\widetilde{A})(I-z\widetilde{A})^{-1}{\bf 1}^\top$ with minimal modulus is strictly larger than 1, and $(\widetilde{\bf A3})$ is satisfied.
		
		Although ${\cal G}(z)$ in (\ref{prob2}) assumes that there is no probability mass at zero, similar to  the continuous-time case ${\cal L}(s)$ given by (\ref{prob1}),  our
		following analysis and results on ${\cal G}(z)$ in (\ref{prob2})  can be directly adopted to the case
		with positive probability mass at zero by just setting the initial probability at the
		absorbing state to be this positive probability mass at zero for the Markov chain we
		have identified. Without losing generality, here we only focus on ${\cal G}(z)$ in (\ref{prob2}).
		
		First we establish the following theorem to transform this discrete-time PH representation problem into a continuous-time PH representation problem.
		
		\begin{theorem}\label{thm:d-c}
			${\cal G}(z)$ given by \eqref{prob2} has an $n$-order discrete-time PH representation $(\widetilde\alpha,\widetilde{A})$ if and only if ${\cal G}(1/(s+1))$ has an $n$-order continuous-time PH representation $(\widetilde \alpha, A)$ with the diagonal entries of $A$ larger than or equal to $-1$.
		\end{theorem}
		\begin{proof}
			Suppose that ${\cal G}(z)$ has an $n$-order discrete-time PH representation $(\widetilde\alpha,\widetilde{A})$, by \eqref{equ:dis-ge},
			\begin{align*}
				{\cal G}(z)=z\widetilde\alpha(I-\widetilde{A}) (I-z\widetilde{A})^{-1}{\bf 1}^\top=\widetilde\alpha(I-\widetilde{A}) (z^{-1}I-\widetilde{A})^{-1}{\bf 1}^\top.
			\end{align*}
			Note that both sides of this equation are analytical in the region $\{z\in {\cal C}:|z|\leq 1\}$ and when $s\in {\cal C}$ and ${\sf Re}(s)\geq 0$, $|1/(s+1)|\leq 1$. Then for $s\in {\cal C}$ and ${\sf Re}(s)\geq 0$,
			\begin{align*}
				{\cal G}(1/(s+1))=\widetilde\alpha(I-\widetilde{A}) ((s+1)I-\widetilde{A})^{-1}{\bf 1}^\top=-\widetilde \alpha(\widetilde{A}-I) (s-(\widetilde{A}-I))^{-1}{\bf 1}^\top.
			\end{align*}
			Let $A=\widetilde{A}-I$ and write $A=(a_{ij})_{n\times n}$. In view of the property of $\widetilde{A}$, we know that $A$ is nonsingular, $A{\bf 1}^\top\leq {\bf 0}$, $a_{ii}\geq -1$ and $a_{ij}\geq 0$ for $i\neq j, 1\leq i,j\leq n$. Therefore, $(\widetilde \alpha, A)$ is an $n$-order continuous-time PH representation.
			
			Conversely, suppose that ${\cal G}(1/(s+1))$ has an $n$-order continuous-time representation $(\alpha,A)$ with the diagonal entries of $A$ larger than or equal to $-1$. That is,
			\begin{align*}
				{\cal G}(1/(s+1))=-\alpha A (sI-A)^{-1}{\bf 1}^\top=-\alpha A ((s+1)I-(A+I))^{-1}{\bf 1}^\top
			\end{align*}
			for $s\in {\cal C}$ and ${\sf Re}(s)\geq 0$.
			Then for $z\in \{z\in {\cal C}:{\sf Re}(z)>0,|z|\leq 1\}$,
			\begin{align*}
				{\cal G} (z)=-\alpha A (z^{-1}I-(A+I))^{-1}{\bf 1}^\top.
			\end{align*}
			Observe that both sides of the above equation are analytical in the region $\{z\in {\cal C}:|z|\leq 1\}$. Using the uniqueness of analytical function, the above equation also holds in the region $\{z\in {\cal C}:|z|\leq 1\}$.
			Let $\widetilde{A}=A+I$. We have
			\begin{align*}
				{\cal G}(z)=\alpha(I-\widetilde{A}) (z^{-1}I-\widetilde{A})^{-1}{\bf 1}^\top=z\alpha(I-\widetilde{A}) (I-z\widetilde{A})^{-1}{\bf 1}^\top.
			\end{align*}
			In view of the property of $A$ and $a_{ii}\geq -1$, we know $(\alpha, \widetilde{A})$ is an $n$-order discrete-time PH representation.
		\end{proof}
		
		Maier \cite{maier1991} shows that ${\cal G}(c/(s+c))$ with any $c>0$ is the LST of a continuous-time PH distribution if ${\cal G}(z)$ is the generating function of a discrete-time PH distribution. This result does not show (i) the minimal representation order relation between these two continuous-time and discrete-time PH distributions; and (ii)  whether ${\cal G}(z)$ is the generating function of a discrete-time PH distribution  given that ${\cal G}(c/(s+c))$ with any $c>0$ is the LST of a continuous-time PH distribution. To completely adopt the results developed in Section \ref{cts_case}, we need the equivalent relation given by Theorem \ref{thm:d-c} between the minimal representations for the continuous-time and discrete-time PH distributions. To answer the minimal representation problem
		for the discrete-time PH distribution specified by (\ref{equ:dis-ge}), the equivalent relation obtained by Theorem \ref{thm:d-c} implies that it is sufficient to consider the continuous-time case with LST given by  ${\cal G}(1/(s+1))$. To apply what we have developed for the continuous-time case in Section  \ref{cts_case}, we have to show ${\cal G}(1/(s+1))$ has the canonical form  of (\ref{prob1}).

		\begin{theorem}\label{prop:d-c}
			For ${\cal G}(z)$ given by \eqref{prob2},	let ${\cal L}_0(s)={\cal G}(1/(s+1))$, $p_0(s)=(s+1)^n\widetilde{p}(1/(s+1))$ and $q_0(s)=(s+1)^n\widetilde{q}(1/(s+1))$. Then
			\begin{itemize}
				\item[{\rm (i)}]
				\begin{align}\label{equ:h}
					{\cal L}_0(s)=\frac{p_0(s)}{q_0(s)}=\frac{\widetilde{p}_1(s+1)^{n-1}+\widetilde{p}_{2}(s+1)^{n-2}+\cdots +\widetilde{p}_n}{(s+1)^n+\widetilde{q}_{2}(s+1)^{n-1}+\cdots+\widetilde{q}_{n}};
				\end{align}	
				\item[{\rm(ii)}] The coefficients of $p_0(s)$ and $q_0(s)$ are real$;$
				\item [{\rm(iii)}] $p_0(s)$ and $q_0(s)$ are coprime$;$
				\item[{\rm(iv)}] ${\cal L}_0(0)=1$ and the root with maximal real part of $q_0(s)$  is real and negative.
			\end{itemize}
		\end{theorem}
		
		\begin{proof}
			(i) and (ii) are straightforward by replacing $z$ with $1/(s+1)$ in \eqref{prob2}. Note that ${\cal L}_0(0)={\cal G}(1)=1$. Further, we can factor $\widetilde{q}(z)$ as the following form:
			\begin{align*}
				\widetilde{q}(z)=(1-\gamma_1z)^{n_1}(1-\gamma_2z)^{n_2}\cdots(1-\gamma_bz)^{n_b}.
			\end{align*}
			Here $\gamma_i\neq 0$. Let $n_{[1,b]}=\sum_{l=1}^bn_l$. If $\widetilde{q}_n\neq 0$, then $n_{[1,b]}=n$; otherwise, $n_{[1,b]}<n$. That is, the set of all distinct roots of $\widetilde{q}(z)$ is $\{1/\gamma_1,\ldots,1/\gamma_b\}$ with $1>\gamma_1>|\gamma_i|$ for $2\leq i\leq b$ from ($\widetilde{\bf A3}$).
			Then
			\begin{align*}
				q_0(s)=\big(s+1\big)^{n- n_{[1,b]}}\big(s-(\gamma_1-1)\big)^{n_1}\big(s-(\gamma_2-1)\big)^{n_2}\cdots\big(s-(\gamma_b-1)\big)^{n_b}.
			\end{align*}
			It follows that $\{\gamma_1-1,\gamma_2-1,\ldots,\gamma_b-1\}$  are all distinct roots of $q(s)$, and $-1$ is also a root when $n_{[1,b]}<n$.   Note that
			${\sf Re}(\gamma_i-1)<\gamma_1-1<0$ for $2\leq i\leq b$ and $\gamma_1-1>-1$. We obtain that the root with maximal real part of $q_0(s)$  is real and negative. Hence, we have (iv).

			It remains to show (iii). If it is not true, suppose that $\gamma_i-1$ is also a root of $p_0(s)$. That is,
			\[
			0=p_0(\gamma_i-1)=\gamma_i^n\widetilde{p}\big(\frac{1}{\gamma_i}\big).
			\]
			$\gamma_i\neq 0$ implies that $\widetilde{p}(1/\gamma_i)=0$. This contradicts that $\widetilde{p}(z)$ and $\widetilde{q}(z)$ are coprime from ($\widetilde{\bf A2}$).
			Then $\gamma_i-1$ cannot be the root of $p_0(s)$. Thus, if $n_{[1,b]}=n$, we complete the roof of (iii).
			
			If $n_{[1,b]}<n$, we need to show that $-1$ is not the root of $p_0(s)$. Note that $n_{[1,b]}<n$ means $\widetilde{q}_n= 0$. This tells us $\widetilde{p}_n\neq 0$. Then $-1$ is not the root of $p_0(s)$ and (iii) still holds.
			Hence, (iii) is proved.
		\end{proof}
		
		With the help of Theorems \ref{thm:d-c} and \ref{prop:d-c}, the discrete-time case with ${\cal G}(z)$ given by
		(\ref{prob2}) is equivalently transformed into the continuous-time case with ${\cal L}_0(s)$ which has the standard form defined by (\ref{prob1}) and satisfies
		({\bf A1})-({\bf A3}). So the discrete-case can be solved by formulating the optimization problem ({\bf OP}) given by (\ref{july-15-2}) with modifying $\xi$ given by \eqref{equ:xi} to be $1$.
		
		\section{Numerical Experiments}
		\label{sec:numerical}
		
		 In this section, we use the {\bf AM} algorithm to conduct the numerical experiments and report the corresponding numerical performance on the minimal representations for both continuous-time and discrete-time cases.
			The numerical experiments lead to the following conclusions: (i) the {\bf AM} algorithm produces  $A$ and $P$ such that $ {\cal F} (A,P)=0$ and they satisfy the constraints in the  non-convex optimization problem {\bf (OP)} given by (\ref{july-15-2}), then we find a minimal PH representation; (ii) when the {\bf AM} algorithm produces a nonzero value, it often does not have a minimal PH representation with the same order as its algebraic degree. Thus, the {\bf AM} algorithm may provide an efficient approach to simultaneously solve the problem whether there exists an algebraic degree minimal PH representation, and the problem how to find the corresponding representation for both continuous-time and discrete-time cases; (iii) choosing different initial $A_0$, we may obtain different minimal PH representations; and (iv) the {\bf AM} algorithm becomes more efficient and effective when ${\cal L}(s)$ only has real poles in terms of the computation time and whether a minimal PH representation can be successfully found.  
			At the same time, we make a comparison between the {\bf AM} algorithm and the {\bf MM} method in Mészáros et al. \cite{telek2013}, and Telek and Horváth \cite{telek2007} by extensive numerical experiments.
		
		First we provide an example to show how to adjust the initial $A_0$ to generate different minimal PH representations.
		\begin{example}
			{\rm	Consider	
				\begin{align*}
					{\cal L}(s)=\frac{1.161}{s+1}-\frac{0.23}{s+2.8+0.4\imath}-\frac{0.23}{s+2.8-0.4\imath}.
				\end{align*}
				If it has a $3$-order PH representation $(\alpha, A)$, the spectrum of $A$ must be $\{-1,-2.8\pm 0.4\imath\}$. Then by Proposition {\rm\ref{prop_Jordan}}, the Jordan form ${\cal J}$ of $A$ is
				\begin{align*}
					\begin{bmatrix}
						-1&0 & 0\\
						0&-2.8&0.4\\
						0&-0.4&-2.8\\
					\end{bmatrix}.
				\end{align*}
				From \eqref{equ_beta} and \eqref{equ:xi} and Theorem {\rm\ref{iff}}, we can calculate
				\begin{align*}
					\beta=(1.161,-0.092,-0.069),
				\end{align*}	
				and $\xi=6.6$. Now we use the {\bf AM} algorithm to solve the optimization problem {\bf (OP)} given by \eqref{july-15-2}.	Let $A_0={\cal J}$. The outcome of the {\bf AM} algorithm is
				\begin{align*}
					P_{\infty} \approx \begin{bmatrix}
						0.5347  &  0.1988  &  0.2665\\
						-0.1520 &   0.2325 &  0.9195\\
						-0.1251 &   -0.0858 &   1.2109
					\end{bmatrix}, \ A_{\infty} \approx \begin{bmatrix}
						-1.1120 &   0.5827 &   0.0033\\
						0.0000 & -2.5666  &  2.5666\\
						0.2246  &  0.0000  & -2.9216
					\end{bmatrix}
				\end{align*}
				and ${\cal F}(A_{\infty},P_{\infty}) \approx 3.6262\times 10^{-9}$. By Theorem {\rm\ref{iff}}, let $\alpha=\beta P_{\infty} \approx (0.6434,0.2154,0.1412)$.	
				We obtain a minimal PH representation $(\alpha,A_{\infty})$ of ${\cal L}(s)$.
				
				If we choose
				\begin{align*}
					\widehat{A}_0=\begin{bmatrix}
						-1&1&0\\
						0&0&0\\
						0&0&0
					\end{bmatrix}.
				\end{align*}
				The outcome of the {\bf AM} algorithm is
				\begin{align*}
					\widehat{P}_{\infty} \approx \begin{bmatrix}
						0.3438  &  0.3351  &  0.3211\\
						0.3568   & 0.8141   &-0.1709\\
						0.0427   & 1.0944   &-0.1371
					\end{bmatrix}, \
					\widehat{A}_{\infty} \approx \begin{bmatrix}
						-2.1582  &  2.0974  &  0.0000\\
						0.0000   &-3.1762   & 0.2547\\
						1.2403   & 0.0253   &-1.2658
					\end{bmatrix}
				\end{align*}
				and ${\cal F}(\widehat{A}_{\infty},\widehat{P}_{\infty}) \approx 4.0136 \times 10^{-9}$. Let $\widehat{\alpha}=\beta \widehat{P}_{\infty} \approx (0.3634,0.2386,0.3980)$.	
				We obtain another minimal PH representation $(\widehat{\alpha},\widehat{A}_{\infty})$ of ${\cal L}(s)$. \hfill$\Box$
			}
		\end{example}
		
		The following example shows that when the  {\bf AM} algorithm
		produces a nonzero value, quite often the minimal algebraic degree PH representation does not exist.
		
		\begin{example}
			{\rm	Let
				\begin{align*}
					{\cal G} (z)=\frac{0.8854z^3-0.5948z^2+0.0294z}{(1-0.5z)(1-0.2z)^2}
				\end{align*}
				We want to identify if it has a $3$-order discrete-time PH representation.
				We can convert it into continuous-time case. From Theorem {\rm\ref{thm:d-c}}, let
				\begin{align*}
					{\cal L}_0(s)={\cal G}\Big(\frac{1}{s+1}\Big)=\frac{0.0294s^2-0.536s+0.32}{(s+0.5)(s+0.8)^2},
				\end{align*}
				Then we only need to verify if ${\cal L}_0(s)$ has a $3$-order PH representation. Here ${\cal J}$ is
				\begin{align*}
					\begin{bmatrix}
						-0.5&0 & 0\\
						0&-0.8&0\\
						0&1&-0.8\\
					\end{bmatrix}.
				\end{align*}
				From \eqref{equ_beta}, we can calculate that
				\begin{align*}
					\beta=(13.23,-9.0316,-3.1984).
				\end{align*}
				Now we use the {\bf AM} algorithm to solve the optimization problem {\bf (OP)} given by \eqref{july-15-2}.
				Regardless of what we choose the initial $A_0$, the  {\bf AM} algorithm never produces an optimal zero-value result.
				
				
				Actually, for this specific numerical example, we can also theoretically prove that ${\cal G}(z)$ has no 3-order discrete-time PH representation! \hfill$\Box$
			}
		\end{example}
		The next example shows that sometimes the initial $A_0$ is critical to the efficiency of the {\bf AM} algorithm ({\bf Steps a-c}).
		\begin{example}\label{examp:dis}
			{\rm Consider
				\begin{align*}
					{\cal L}(s)=\frac{0.9421}{s+1}-\frac{0.0942}{(s+1.2)^2}+\frac{0.1884}{s+1.3}-\frac{0.0471}{(s+1.3)^3}.
				\end{align*}
				If it has a 6-order continuous-time PH representation $(\alpha, A)$, the spectrum of $A$ must be $\{-1,-1.2,-1.3\}$, and their multiplicities are
				$1,2,3$ respectively. Then the Jordan form ${\cal J}$ of $A$ is
				\begin{align*}
					\begin{bmatrix}
						-1& & & & &\\
						0& -1.2& & & &\\
						& 1&-1.2& & &\\
						& &0 & -1.3& &\\
						& & & 1& -1.3 &\\
						& & & &1&-1.3
					\end{bmatrix}.
				\end{align*}
				From \eqref{equ_beta} and \eqref{equ:xi} and Theorem \ref{iff}, we can calculate
				\begin{align*}
					\beta=(0.9421,0.1798,-0.0904,-0.0391,0.0080,-0.0004),
				\end{align*}	
				and $\xi=1+2\times 1.2+3\times 1.3=7.3$. Now we use the {\bf AM} algorithm to solve the optimization problem {\bf (OP)} given by (\ref{july-15-2}).	Let $A_0={\cal J}$. The outcome of the {\bf AM} algorithm is
				\begin{align*}
					&P_{\infty} \approx \begin{bmatrix}
						0.5918  &  0.2821 &   0.0559  &  0.0516  &  0.0148 &   0.0038 \\
						-0.0018  &  0.5788 &   0.0001  &  0.3819  &  0.0376  &  0.0035 \\
						-0.0097  &  4.0031 &   0.5785  & -3.6285  &  0.0188  &  0.0378 \\
						-0.0020 &   -0.0037 &  0.0000  &  1.0066  & -0.0008  & 0.0000 \\
						-0.0079  &  -0.0199 &  -0.0021  & 0.0248  &  1.0065  & -0.0013 \\
						-0.0279  &  0.0157&  -0.0221  & 0.0036  & 0.0239  &  1.0069
					\end{bmatrix},
				\end{align*}
				\begin{align*}
					&A_{\infty} \approx \begin{bmatrix}
						-1.0004 &   0.0005 &   0.0188  &  0.0006  &  0.0008  &  0.0008\\
						0.0002  & -1.2005  &  0.0000   & 0.0004   & 0.0004   & 0.0006\\
						0.0024  &  1.0059  & -1.1995   & 0.0026   & 0.0004   & 0.0085\\
						0.0006  &  0.0004  &  0.0000   &-1.2992   & 0.0000   & 0.0000\\
						0.0003  &  0.0003  &  0.0003   & 1.0001   &-1.2995   & 0.0000\\
						0.0005  &  0.0008  &  0.0006   & 0.0009   & 0.9997   &-1.3011\\
					\end{bmatrix},
				\end{align*}
				and ${\cal F}(A_{\infty},P_{\infty}) \approx 6.2781 \times 10^{-10}$. By  Theorem {\rm\ref{iff}},  let $$\alpha=\beta P_{\infty} \approx (0.5581,0.0079,0.0004,0.4061,0.0271,0.0004).$$	
				We obtain a minimal PH representation $(\alpha,A_{\infty})$ of ${\cal L}(s)$.
				
				But if we choose $A_0=-\xi I$, the outcome ${\cal F}_{\infty} \approx 0.2114$. In this case, we can not find a 6-order PH representation. \hfill$\Box$
			}	
		\end{example}
		
		Now we demonstrate the effectiveness and efficiency of the {\bf AM} algorithm through a series of numerical experiments by the following procedure:
			\begin{itemize}
				\item[{\sf (P1)}] Generate an $n$-order continuous-time PH distribution $(\alpha, A)$ randomly and compute its LST
				${\cal L}(s)$;
				\item[{\sf (P2)}] Select $(\alpha, A)$ from {\sf (P1)} such that
				its algebraic degree is also $n$;
				\item[{\sf (P3)}]  Calculate ${\cal J}$ by (\ref{prop-1-1}) in Proposition \ref{prop_Jordan}, $\beta$ by \eqref{equ_beta} from Proposition \ref{prop_beta}, and $\xi$ by $\xi=-\sum_{i=1}^n{\cal J}_{ii}$;
				\item[{\sf (P4)}] Implement the {\bf AM} algorithm to solve the optimization problem {\bf (OP)} in (\ref{july-15-2}) for ${\cal J}$, $\beta$ and $\xi$.
			\end{itemize}
			
			In {\sf (P1)}, the choice of $\alpha=(\alpha_1,\alpha_2,\dots,\alpha_n)$ follows the rule: $\alpha_1$ is sampled from uniform distribution ${\cal U}(0,1)$; $\alpha_i$ is sampled from ${\cal U}(0,1-\sum_{j=1}^{i-1}\alpha_j)$ for $2\leq i\leq n-1$, and $\alpha_n=1-\sum_{j=1}^{n-1}\alpha_j$. For the choice of $A=(a_{ij})_{n\times n}$, the off-diagonal elements of $A$ are sampled from independent uniform distributions ${\cal U}(0,c)$ for some positive number $c$; the diagonal element $a_{ii}=-\sum_{j=1,j\neq i}^n a_{ij}-\theta$, and $\theta$ is also sampled from ${\cal U}(0,c)$. Hence, we randomly generate a continuous-time PH distribution  $(\alpha, A)$ with the algebraic degree $n$. For each $n$ with $3\leq n\leq 7$, we first repeat the above procedure to randomly generate 1000 continuous-time PH distributions $(\alpha, A)$ with the algebraic degree $n$.
			
			We then employ the {\bf AM} algorithm by setting $A_0={\cal J} + {\bf 1}^\top {\bf 1}- I$. The iterations will be terminated at step-$k$ if $|{\cal F}(A_{k+1},P_{k})-{\cal F}(A_{k},P_{k})|\leq 10^{-13}$. If the outcome objective value ${\cal F}(A_{k},P_{k})$ is less than $n^2\times 10^{-10}$, then we consider \eqref{equ_sys} to be solvable and the {\bf AM} algorithm provides an $n$-order PH representation. Here we solve the convex optimization problems  {\bf (OP[$P$])} and {\bf (OP[$A$])} in each step using Matlab's ``quadprog" function. Table \ref{table_am} below summarizes the performance over randomly generated 1000 PH distributions for each $n$.  These distributions are divided into two groups based on whether $\mathcal{L}(s)$ has complex poles: the ``Real Case" for distributions without complex poles, and the ``Complex Case" for those with complex poles. The table shows the number of successful experiments and the average number of iterations for both cases, denoted by ``Real Succ.", ``Real Ave. Iter.", ``Complex Succ.", and ``Complex Ave. Iter." respectively, for brevity. Additionally, the table provides the total number of successful experiments, the average number of iterations, and the average runtime, denoted by ``Total Succ.", ``Ave. Iter." and ``Ave. Runtime".

			\begin{table}[h!]
				\begin{center}
					\caption{Outcome of the {\bf AM} algorithm}\label{table_am}
					\begin{tabular}{|c|c|c|c|c|c|} 
						\hline
						$n$ & 3 & 4 & 5 &  6 & 7  \\
						\hline
						Real Succ. & 812/812 & 566/566 & 350/350& 181/181& 90/90\\
						\hline
						Real Ave. Iter. & 14.04& 22.54 & 39.08& 34.14& 90.66\\
						\hline
						Complex Succ. & 188/188& 434/434& 649/650& 804/819 & 879/910\\
						\hline
						Complex Ave. Iter.& 115.64& 90.25 & 130.57 & 124.85 & 211.89\\
						\hline
						Total Succ. & 1000/1000 & 1000/1000 & 999/1000 & 985/1000 & 969/1000 \\
						\hline
						Ave. Iter. & 33.14& 51.93& 98.52& 108.18& 200.63\\
						\hline
						Ave. Runtime & 0.0317s& 0.191s& 0.178s& 1.357s& 1.984s\\
						\hline
					\end{tabular}
				\end{center}
			\end{table}
			Few observations are made from Table \ref{table_am}: (a) In each experiment, simply keeping $A_0$ fixed at $A_0={\cal J} + {\bf 1}^\top {\bf 1}- I$, the {\bf AM} algorithm can effectively obtain a minimal PH representation in most cases; (b) It can be seen that all failures occur in the ``Complex Case", and the number of iterations in the ``Complex Case" is significantly higher than that in the ``Real Case". Therefore, the {\bf AM} algorithm is more efficient for the ``Real Case" compared to the ``Complex Case"; 
			(c) As $n$ increases, the average number of iterations and the average runtime tend to increase.

			Finally we compare the {\bf AM} algorithm with the {\bf MM}  method proposed by M\'esz\'aros et al. \cite{telek2013}, and Telek and Horv\'ath \cite{telek2007}.  For a given ${\cal L}(s)$ defined in \eqref{prob1}, their {\bf MM} method starts with a matrix representation $(v,H)$ of ${\cal L}(s)$, where $v\in {\cal R}^n$ and $H\in {\cal R}^{n\times n}$. We say $(v,H)$ is a {\it matrix representation} of ${\cal L}(s)$ if $-vH(sI-H)^{-1}{\bf 1}^\top={\cal L}(s)$. Note that the components of $v$ and the off-diagonal elements of $H$ may not be nonnegative, and the diagonal elements of $H$ may not be negative. To address this incorrectness issue for $H$ to be a generator, they introduce different parameterized error functions to penalize the terms with incorrect positive or negative signs.  In each step, they choose an elementary transformation that minimizes the error function. Keeping applying such elementary transformations to the initial matrix representation in each step, their {\bf MM} method may or may not generate a PH representation. To make these two comparable, the initial matrix representation for the {\bf MM} method is also set to $(\beta, {\cal J})$ defined by \eqref{equ_beta}. From \eqref{july-15-2}, this initial setting makes both algorithms to utilize the same information of ${\cal L}(s)$, namely $\beta$ and ${\cal J}$.
			
			To evaluate their performance, we follow the above procedure {\sf (P1)}-{\sf (P4)}. In particular, to generate the following three cases studied by  M\'esz\'aros et al. \cite{telek2013}, and Telek and Horv\'ath \cite{telek2007}, the choices for the off-diagonal elements of $A=(a_{ij})_{n\times n}$ in {\sf (P1)} are modified by  
			\begin{itemize}
				\item[1.] {\it Balanced Case} ({\sf BC}): Its off-diagonal elements are sampled from independent uniform distributions ${\cal U}(0,c)$; 
				\item[2.] {\it Sparse Case} ({\sf SP}):  With probability $p$, its off-diagonal elements are set to zero, and with probability $1-p$, they are sampled from independent uniform distributions ${\cal U}(0,c)$; 
				\item[3.] {\it Stiff Case} ({\sf ST}): With probability $p$, its off-diagonal elements are sampled from independent uniform distributions ${\cal U}(0,1000c)$, and with probability $1-p$, they are sampled from independent uniform distributions ${\cal U}(0,c)$.
			\end{itemize}
			
			For the sparse case, we consider $p=0.5, 0.7, 0.9$, and for the stiff case, consider $p=0.3, 0.5, 0.7$. In sum,  three cases with seven  scenarios are considered. For each scenario, 100 continuous-time PH distributions $(\alpha, A)$ with the algebraic degree $n$ are randomly generated for $3\leq n\leq 6$. For the {\bf AM} algorithm starting with $A_0={\cal J} + {\bf 1}^\top {\bf 1}- I$, the iterations will be terminated at step-$k$ if $|{\cal F}(A_{k+1},P_{k})-{\cal F}(A_{k},P_{k})|\leq 10^{-13}$. If the outcome objective value ${\cal F}(A_{k},P_{k})$ is less than $n^2\times 10^{-10}$, then we consider \eqref{equ_sys} to be solvable and the {\bf AM} algorithm provides an $n$-order PH representation.
			For the {\bf MM} method starting with $A_0={\cal J} + {\bf 1}^\top {\bf 1}- I$, its several modified versions are developed in M\'esz\'aros et al. \cite{telek2013}, and Telek and Horv\'ath \cite{telek2007}. Here we report the best results over all its modified versions.  When executing the {\bf MM} method, the number of iterations should be determined to terminate the computation, and the error functions may be adjusted to improve the performance. To ensure that the two algorithms are compared within the same error range, the tolerance is set to $10^{-5}$. That is, if the error in the incorrect positive or negative signs of the outcome is within the tolerance, we consider that the {\bf MM} method has successfully provided a PH representation. 

			\begin{table}[h!]
				\begin{center}
					\caption{Outcome of the {\bf AM} and {\bf MM}}\label{table_amm}
					\begin{tabular}{|l|c|c|c|c|c|c|c|}
						\hline
						&{\sf BC}&{\sf SP}(0.5)&{\sf SP}(0.7)&{\sf SP}(0.9)&{\sf ST}(0.3)&{\sf ST}(0.5)&{\sf ST}(0.7)\\
						\hline
						$n=3$ ({\bf AM}) & 100 & 99& 99 & 97 & 93&86&74\\
						\hline
						$n=3$ ({\bf MM}) & 100 & 99 & 96 & 93 & 90&79&70\\
						\hline
						$n=3$ (Total) & 100 & 100 & 100 & 99&95 &89&77\\
						\hline
						$n=4$ ({\bf AM}) & 100& 94 & 93 & 86 &97 &56&55\\
						\hline
						$n=4$ ({\bf MM}) & 100 & 96 & 97 & 86 & 95&54&39\\
						\hline
						$n=4$ (Total) & 100 & 98 & 100& 94 &98 &64&56\\
						\hline
						$n=5$ ({\bf AM}) & 100 & 93& 84& 81 &66 &37&32\\
						\hline
						$n=5$ ({\bf MM}) & 100 & 91& 82 & 89 &59 &33&28\\
						\hline
						$n=5$ (Total) & 100 & 97& 90& 92 &70 &41&36\\
						\hline
					\end{tabular}
				\end{center}
			\end{table}
			Table \ref{table_amm} summarizes the number of successful experiments separately obtained by the {\bf AM} algorithm and {\bf MM} method for the randomly generated 100 PH distributions for each $n$ and each scenario. ``Total" represents the number of successes where either the {\bf AM} algorithm or {\bf MM} method can successfully provide an $n$-order PH representation. The table shows that the two algorithms exhibit almost similar performance. In the ``Balanced Case", both algorithms perform well. When the generator becomes sparser or contains elements with big difference, however, both of them don't perform well. 
			
			Actually, the two algorithms operate within two fundamentally different frameworks. In the {\bf AM} algorithm, at each step $k$, the constraints ensure that $\beta P_k$ remains an distribution and $A_k$ is a generator while minimizing the objective function aims to enhance the similarity between $A_k$ and $\mathcal{J}$, thereby aligning the  Laplace transform of $(\beta P_k, A_k)$ as closely as possible with $\mathcal{L}(s)$. In contrast, for the {\bf MM} method, the elementary transformation at step $k$ with output $(\alpha_k, A_k)$  keeps the  similarity between $A_k$ and $\mathcal{J}$, ensuring that the Laplace transform of $(\alpha_k, A_k)$ is $\mathcal{L}(s)$, while the error functions attempt to ensure that $\alpha_k$ is an initial distribution and $A_k$ is a generator.  When both of them yield a solution, the {\bf AM} algorithm guarantees that the outcome $(\beta P_{\infty}, A_{\infty})$ is a phase-type representation with any error reflected in the difference between the Laplace transform of $(\beta P_{\infty}, A_{\infty})$ and $\mathcal{L}(s)$. In contrast, the outcome $(\alpha_{\infty}, A_{\infty})$ produced by the {\bf MM} method implies that its Laplace transform is $\mathcal{L}(s)$, however, due to the presence of a tolerance, $(\alpha_{\infty}, A_{\infty})$ may not be a phase-type representation. Moreover, from the robustness and simplicity perspective, the {\bf AM} algorithm demonstrates more robust and simpler across different problems, as there is no need to properly and dynamically tune the error functions and set the step-length to determine the similarity transformation matrix which are critical to the {\bf MM} method. The following example studied by  Telek and Horv\'ath \cite{telek2007} illustrates this in more concrete way.
		
		\begin{example}\label{examp:robustness} Consider 
			\begin{align*}
				\beta=(0.2,0.3,0.5) \quad \mbox{and} \quad {\cal J}=\begin{bmatrix}
					-1&0&0\\
					0&-3&h\\
					0&-h&-3
				\end{bmatrix}.
			\end{align*}
			Let ${\cal L}(s)=-\beta {\cal J}(sI-{\cal J})^{-1}{\bf 1}^\top.$
			They point out that if $h\leq 0.552748375$, ${\cal L}(s)$ has a 3-order continuous-time PH representation. As $h$ tends to 0.552748375, when using the {\bf MM} algorithm, modifications to the parameters of the error function are necessary to obtain a 3-order PH representation. However, with the {\bf AM} algorithm, we can always choose $A_0={\cal J} + {\bf 1}^\top {\bf 1}- I$ to obtain a 3-order PH representation. For instance, if $h=0.545$, the outcome of the {\bf AM} algorithm is
			\begin{align*}
				P_{\infty}\approx \begin{bmatrix}
					0.3951722181  &  0.3536663051  &  0.2511614767\\
					-0.1127018658 &   0.8684120568 &  0.2442898089\\
					-0.0904477646 &   1.2540523021 &   -0.1636045374
				\end{bmatrix}, \\
				A_{\infty}\approx\begin{bmatrix}
					-1.1627675785 &   0.0001744592 &   1.1623247607\\
					0.1815466311  & -3.0084620976 &  0.0000113592\\
					0.0004485118  &  2.8278906161  & -2.8287865586
				\end{bmatrix},
			\end{align*}
			and ${\cal F}(A_{\infty},P_{\infty})\approx1.5411\times 10^{-10}$. By  Theorem {\rm\ref{iff}},  $(\alpha,A_\infty)$ is a 3-order PH distribution with $\alpha=\beta P_{\infty}\approx(0.0000000016,0.9582830291,0.0417169693)$.
			If $h=0.552$, the outcome of the {\bf AM} algorithm is
			\begin{align*}
				P_{\infty}\approx\begin{bmatrix}
					0.3955359095  &  0.2431603865  &  0.3613037039\\
					-0.1129899070 &   0.2284325823 &  0.8845573247\\
					-0.0904204030 &   -0.1930954789 &   1.2835158819
				\end{bmatrix},
			\end{align*}
			\begin{align*}
				A_{\infty}\approx\begin{bmatrix}
					-1.1630039882&   1.1627976802 &   0.0000665242\\
					0.0005374655  & -2.8914646767  &  2.8906923062\\
					0.1780859324  &  0.0000037036  & -2.9455323248
				\end{bmatrix},
			\end{align*}
			and ${\cal F}(A_{\infty},P_{\infty})\approx2.0198\times 10^{-12}$. By  Theorem {\rm\ref{iff}},  $(\alpha,A_\infty)$ is a 3-order PH distribution with $\alpha=\beta P_{\infty}\approx(0.0000000083,0.0206141126,0.9793858791)$.
		\end{example}
		
		\section{Concluding Remarks}\label{sec:conclude}
		
		In this paper we investigate the minimal representation problem of PH distributions starting with their LST. The if and only if condition is established to verify whether a PH distribution has the same order representation as its algebraic degree. To this end, we develop a unique expression for LTS in terms of the Jordan form. The unique expression helps us equivalently transform the minimal representation problem into the solution existence problem for the set of linear and quadratic equations. The latter problem is further equivalently transformed into  the non-convex optimization problem, and the alternating minimization ({\bf AM}) algorithm is then developed to solve this optimization problem, the proposed algorithm is proved to be convergent for any initial value, and numerically demonstrated to be quite effective. Moreover, the approach we develop for the continuous-time PH distributions can be applied to the discrete-time PH distributions by establishing the equivalence between the LST of continuous-time PH distributions and the generating function of discrete-time distributions.

		The next research topic is to consider the case in which the order of its PH representation is not the same as its algebraic degree.
		The challenge for this case is that from the LST's poles, how to identify the other poles. This is very subtle from the discussions conducted in    Dehon and Latouche \cite{dehon1982}, and  O'Cinneide \cite{o1990characterization}.

		\appendix
		\section{Proof of Proposition \ref{prop_Jordan}}
		This proposition may be hidden in some textbook, for the completeness, we present a proof. 
		\begin{proof} First note that from (\ref{prop-1-10}), 
			\begin{eqnarray}
				&& (sI-J)^{-1}= {\sf diag}\Big(
				(sI-J^{(1)}(\lambda_1))^{-1},\cdots,(sI-J^{(\ell)}(\lambda_{\ell}))^{-1},\label{inverse} \\
				&& \ \ \ \ \ \ \ \ \ \ \ \ \ \ \ \ \ \ \ \ \ \ \ \ \ \ (sI-J^{(\ell+1)}(\mu_{\ell+1},\omega_{\ell+1}))^{-1},
				\cdots,  (sI-J^{(k)}(\mu_{k},\omega_{k}))^{-1}
				\Big). \nonumber
			\end{eqnarray}
				
				We now prove the proposition for $1\leq i\leq \ell$. 	
				Assume $J^{(i)}(\lambda_i)\neq {\cal J}_{n_i}(\lambda_i)$. Suppose
				\begin{align}\label{prop-1-proof-1}
					J^{(i)}(\lambda_i)=\begin{bmatrix}
						{\cal J}_{m_i}(\lambda_i)& \\
						&{\cal J}_{n_i-m_i}(\lambda_i)
					\end{bmatrix} \ \mbox{for some positive integer $m_i$ with $m_i<n_i$}.
					\end{align}
					Then
					\begin{align*}
						\Big(sI-J^{(i)}(\lambda_i)\Big)^{-1}=\begin{bmatrix}
							\Big(sI-{\cal J}_{m_i}(\lambda_i)\Big)^{-1}&\\
							&\Big(sI-{\cal J}_{n_i-m_i}(\lambda_i)\Big)^{-1}
						\end{bmatrix}.
					\end{align*}
					This, by (\ref{inverse-J-1}), implies that the largest degree of $(s-\lambda_i)^{-1}$ among the $n_i^2$ entries of $\big(sI-J^{(i)}(\lambda_i)\big)^{-1}$ is
					$\max\{m_i,n_i-m_i\}$, which is strictly smaller than $n_i$. Then, by (\ref{inverse}),  the largest degree of $(s-\lambda_i)^{-1}$ among the $n^2$ entries of $(sI-J)^{-1}$ is also strictly smaller than $n_i$. In view of \eqref{equ_Jordanform}, the rational function $-\alpha A(sI-A)^{-1}{\bf 1}^\top$ is a linear combination of entries of $(sI-J )^{-1}$. In this case, the multiplicity of pole $\lambda_i$ in rational function $-\alpha A (sI-A)^{-1}{\bf 1}^\top$ is strictly smaller than $n_i$. Thus such $A$ cannot make \eqref{equ_Lapl} hold. Therefore, $J^{(i)}(\lambda_i)$ cannot have two sub-blocks given by (\ref{prop-1-proof-1}). Similarly, we can also prove that $J^{(i)}(\lambda_i)$ cannot have more than two sub-blocks with each sub-block being form ${\cal J}_m(\lambda_i)$. Thus we complete the proof of the first part of the proposition.
					
					For the second part of the proposition, the proof is similar. Assume that $J^{(j)}(\mu_j,\omega_{j})\neq {\cal J}_{n_j}(\mu_j,\omega_{j})$.  Suppose then
					\begin{align}\label{prop-1-proof-2}
						J^{(j)}(\mu_j,\omega_{j})=\begin{bmatrix}
							{\cal J}_{m_j}(\mu_j,\omega_{j})& \\
							&{\cal J}_{n_j-m_j}(\mu_j,\omega_{j})
						\end{bmatrix} \ \ \mbox{for some $m_j$ with $0<m_j<n_j$.}
						\end{align}
						A straightforward calculation yields that
						\begin{align*}
							\Big(sI-J^{(j)}(\mu_j,\omega_{j})\Big)^{-1}=\begin{bmatrix}
								\Big(sI-{\cal J}_{m_j}(\mu_j,\omega_{j})\Big)^{-1}&\\
								&\Big(sI-{\cal J}_{n_j-m_j}(\mu_j,\omega_{j})\Big)^{-1}
							\end{bmatrix}.
						\end{align*}
						Then by (\ref{inverse-J-2}), the largest degree of $\Big((s-\mu_j)^2+\omega_j^2\Big)^{-1}$ in $(sI-J)^{-1}$ is
						$\max\{m_j,n_j-m_j\}$, which is strictly smaller than $n_j$. It follows that $\mu_j\pm \imath \omega_{j}$ could not be $n_j$ multiple poles of ${\cal L}(s)$.
						Thus \eqref{equ_Lapl} cannot be true. This leads to a contradiction. Therefore, $J^{(j)}(\mu_j,\omega_j)$ cannot have two sub-blocks given by (\ref{prop-1-proof-2}). Similarly, we can also prove that $J^{(j)}(\mu_j,\omega_j)$ cannot have more than two sub-blocks with each sub-block being form ${\cal J}_m(\mu_j,\omega_j)$. Thus we complete the proof of the second part of the proposition, and the whole proof of the proposition is then completed.
					\end{proof}
					
					\section{Proof of Proposition \ref{prop_P}}
					Telek and Horv\'ath \cite{telek2007} prove a similar result for $(-A)^{-1}$, see their Theorem 5. Here we adopt their idea to prove the proposition about $A$ itself.
					\begin{proof}  For each real pole $\lambda_i$, define a real $(n_i\times n_i)$-matrix
						\begin{eqnarray}
							F_{n_i}(\lambda_i)=\begin{bmatrix}
								x_{i1}& & &  \\
								x_{i2}& x_{i1}& &  \\
								\vdots&\ddots & \ddots& \\
								x_{in_i}&\cdots  & x_{i2}&x_{i1}
							\end{bmatrix}.
							\label{prop-2-5}
						\end{eqnarray} 
						and for each pair of conjugate complex poles $\mu_{j}\pm \imath\omega_{j}$, define a real $((2n_j)\times(2n_j))$-matrix
						\begin{eqnarray}
							F_{n_j}(\mu_j,\omega_j)=\begin{bmatrix}
								D_{j1}& & &  \\
								D_{j2}& D_{j1}& &  \\
								\vdots&\ddots & \ddots& \\
								D_{jn_j}&\cdots  & D_{j2}&D_{j1}
							\end{bmatrix}  \mbox{with} \ D_{jl}=\begin{bmatrix}
								u_{jl}&v_{jl}\\
								-v_{jl}& u_{jl}
							\end{bmatrix},  \ 1\leq l\leq n_j.
							\label{prop-2-6}
						\end{eqnarray}
						It is straightforward to verify that
						\begin{eqnarray*}
							&& {\cal J}_{n_i}(\lambda_i) \cdot F_{n_i}(\lambda_i)=F_{n_i}(\lambda_i)\cdot  {\cal J}_{n_i}(\lambda_i) \ \
							\mbox{for $1\leq i\leq \ell$};\label{prop-2-1}\\
							&& {\cal J}_{n_j}(\mu_j,\omega_j)\cdot F_{n_j}(\mu_j,\omega_j) = F_{n_j}(\mu_j,\omega_j)\cdot  {\cal J}_{n_j}(\mu_j,\omega_j)
							\ \ \mbox{for $\ell+1\leq j\leq k$}.\label{prop-2-2}
						\end{eqnarray*}
						Let
						\begin{eqnarray}
							F={\sf diag}\Big(F_{n_1}(\lambda_1),\cdots,F_{n_{\ell}}(\lambda_\ell),F_{n_{\ell+1}}(\mu_{\ell+1},\omega_{\ell+1}),\cdots,F_{n_k}(\mu_k,\omega_k)\Big).
							\label{prop-2-3}
						\end{eqnarray}
						Hence, for ${\cal J}$ given in (\ref{prop-1-1}), we have
						\begin{eqnarray}
							&& {\cal J}\cdot F=  F\cdot {\cal J}.\label{prop-2-1}
						\end{eqnarray}
						For any nonsingular $\widetilde{P}$ such that $A=\widetilde{P}^{-1}{\cal J}\widetilde{P}$, if we can properly choose
						$F_{n_i}(\lambda_i)$ and $F_{n_j}(\mu_j,\omega_j)$ such that
						\begin{eqnarray}
							\mbox{matrix $F$ defined by (\ref{prop-2-3}) is nonsingular and} \ \  F\widetilde{P}{\bf 1}^\top={\bf 1}^\top.
							\label{prop-2-2}
						\end{eqnarray}
						Thus by (\ref{prop-2-1}), we have
						\[
						A=\widetilde{P}^{-1}{\cal J}\widetilde{P}=\widetilde{P}^{-1}F^{-1}F{\cal J}\widetilde{P}=\widetilde{P}^{-1}F^{-1}{\cal J}F\widetilde{P}
						=(F\widetilde{P})^{-1}{\cal J}(F\widetilde{P}).
						\]
						Let $P=F\widetilde{P}$ and then we complete the proof. So it remains to construct
						$F_{n_i}(\lambda_i)$ $(1\leq i\leq \ell)$, and $F_{n_j}(\mu_j,\omega_j)$
						$(\ell+1\leq j\leq k)$
						such that \eqref{prop-2-2} holds.
						
						If $\widetilde{P}{\bf 1}^\top = {\bf 1}^\top$, let $F=I$. Otherwise, taking $\widetilde{P}$ into ${\cal J}$-blocked  partition,
						\begin{align*}
							\widetilde{P}=\left[\begin{array}{ccc}
								\widetilde{P}(\lambda_1)\\
								\vdots\\
								\widetilde{P}(\lambda_{\ell})\\
								\widetilde{P}(\mu_{\ell+1},\omega_{\ell+1})\\
								\vdots\\
								\widetilde{P}(\mu_k,\omega_k)
							\end{array}\right].
							\ \ \mbox{Then}\ \
							F\widetilde{P}=\left[
							\begin{array}
								{ccc}
								F_{n_1}(\lambda_1)\widetilde{P}(\lambda_1)\\
								\vdots  \\
								F_{n_\ell}(\lambda_{\ell})\widetilde{P}(\lambda_\ell)\\
								F_{n_{\ell+1}}(\mu_{\ell+1},\omega_{\ell+1})\widetilde{P}(\mu_{\ell+1},\omega_{\ell+1})\\
								\vdots \\
								F_{n_k}(\mu_k,\omega_k)\widetilde{P}(\mu_k,\omega_k)
							\end{array}
							\right].
						\end{align*} 
						
						We first construct $F_{n_i}(\lambda_i)$ such that
						\begin{eqnarray}
							F_{n_i}(\lambda_i)\widetilde{P}(\lambda_i){\bf 1}^\top={\bf 1}^\top \ \mbox{ for } \ 1\leq i \leq \ell.
							\label{prop-2-4}
						\end{eqnarray}
						Note that in (\ref{prop-2-4}), the dimension of ${\bf 1}^\top$ on the left-hand side is $n$ while on the right hand side is $n_i$.
						With some notation abuse, let $\Big(\widetilde{P}(\lambda_i){\bf 1}^\top\Big)_r$ represent its $r$-th component with $1\leq r\leq n_i$. Note that, by the definition $F_{n_i}(\lambda_i)$ given by (\ref{prop-2-5}),
						\begin{small}
							\begin{align*}
								F_{n_i}(\lambda_i)\widetilde{P}(\lambda_i){\bf 1}^\top=
								\left(\begin{array}{llcl}
									x_{i1}\Big(\widetilde{P}(\lambda_i){\bf 1}^\top\Big)_1\\
									x_{i2}\Big(\widetilde{P}(\lambda_i){\bf 1}^\top\Big)_1+x_{i1}\Big(\widetilde{P}(\lambda_i){\bf 1}^\top\Big)_2\\
									\vdots\\
									x_{in_i}\Big(\widetilde{P}(\lambda_i){\bf 1}^\top\Big)_1+x_{i,n_i-1}\Big(\widetilde{P}(\lambda_i){\bf 1}^\top\Big)_2+\cdots+x_{i1}\Big(\widetilde{P}(\lambda_i){\bf 1}^\top\Big)_{n_i}
								\end{array}\right).
							\end{align*}
						\end{small}
						Applying $A=\widetilde{P}^{-1}{\cal J}\widetilde{P}$, from \eqref{equ_Lapl},
						\begin{align*}
							{\cal L}(s)&=\frac{p(s)}{q(s)}=-\alpha A(sI-A)^{-1}{\bf 1}^\top=-\alpha \widetilde{P}^{-1} {\cal J}(sI-{\cal J})^{-1}\widetilde{P}{\bf 1}^\top.
						\end{align*}
						Specifically, by Proposition \ref{prop_Jordan},
						\begin{small}
							\begin{align*}
								(sI-{\cal J})^{-1}\widetilde{P}{\bf 1}^\top=
								\left(\begin{array}{ccc}
									\Big(sI- {\cal J}_{n_1}(\lambda_1)\Big)^{-1}\widetilde{P}(\lambda_1){\bf 1}^\top\\
									\vdots\\
									\Big(sI- {\cal J}_{n_\ell}(\lambda_{\ell})\Big)^{-1}\widetilde{P}(\lambda_{\ell}){\bf 1}^\top\\
									\Big(sI- {\cal J}_{n_{\ell+1}}(\mu_{\ell+1},\omega_{\ell+1})\Big)^{-1}\widetilde{P}(\mu_{\ell+1},\omega_{\ell+1}){\bf 1}^\top\\
									\vdots\\	
									\Big(sI- {\cal J}_{n_k}(\mu_k,\omega_k)\Big)^{-1}\widetilde{P}(\mu_k,\omega_k){\bf 1}^\top
								\end{array}\right).
							\end{align*}
						\end{small}
						Recall the expression of $\Big(sI- {\cal J}_{n_i}(\lambda_i)\Big)^{-1}$, $(s-\lambda_i)^{-n_i}$ only appears in its first column. This implies that $\Big(\widetilde{P}(\lambda_i){\bf 1}^\top\Big)_1$ is not zero. Otherwise, $\lambda_i$ could not be $n_i$ multiple pole of rational function $-\alpha A(sI-A)^{-1}{\bf 1}^\top$. With the help of this observation, we can uniquely determine $F_{n_i}(\lambda_i)$ such that $F_{n_i}(\lambda_i)\widetilde{P}(\lambda_i){\bf 1}^\top={\bf 1}^\top$ in the following way:
						\begin{align*}
							&x_{i1}=1\bigg/\Big(\widetilde{P}(\lambda_i){\bf 1}^\top\Big)_1;\\
							&x_{ir}=\left[1-\sum_{l=1}^{r-1}x_{il}\Big(\widetilde{P}(\lambda_i){\bf 1}^\top\Big)_{r+1-l}\right] \bigg/ \Big(\widetilde{P}(\lambda_i){\bf 1}^\top\Big)_1, \quad 2\leq r \leq n_i.
						\end{align*}
						Note that $x_{i1}$ is not zero. We know $F_{n_i}(\lambda_i)$ is nonsingular and satisfies (\ref{prop-2-4}).
						
						Using the same method, we now construct $F_{n_j}(\mu_j,\omega_j)$ such that
						\begin{align}\label{equ:cec}
							F_{n_j}(\mu_j,\omega_j)\widetilde{P}(\mu_j,\omega_j){\bf 1}^\top={\bf 1}^\top, \ \ell+1\leq j \leq k.
						\end{align}
						Here the dimension of ${\bf 1}^\top$ on the left hand side is $n$ and that on the right hand side is $2n_j$.
						From the structure feature of ${\cal J}$-blocked partition of $\widetilde P$, the left-hand side of (\ref{equ:cec}) is
						\begin{small}
							\begin{align}\label{equ_complex}
								\left(\begin{array}{llll}
									D_{j1}\widetilde{P}^{(1)}(\mu_j,\omega_j){\bf 1}^\top\\
									D_{j2}\widetilde{P}^{(1)}(\mu_j,\omega_j){\bf 1}^\top+D_{j1}\widetilde{P}^{(2)}(\mu_j,\omega_j){\bf 1}^\top\\
									\vdots\\
									D_{jn_j}\widetilde{P}^{(1)}(\mu_j,\omega_j){\bf 1}^\top+D_{j,n_j-1}\widetilde{P}^{(2)}(\mu_j,\omega_j){\bf 1}^\top+\cdots+D_{j1}\widetilde{P}^{(n_j)}(\mu_j,\omega_j){\bf 1}^\top
								\end{array}\right)
							\end{align}
						\end{small}
						and
						\begin{small}
							\begin{align}
								D_{jr}\widetilde{P}^{(1)}(\mu_j,\omega_j){\bf 1}^\top&=
								\left(\begin{array}{cc}
									u_{jr}\big(\widetilde{P}^{(1)}(\mu_j,\omega_j){\bf 1}^\top\big)_1+v_{jr}\big(\widetilde{P}^{(1)}(\mu_j,\omega_j){\bf 1}^\top\big)_2\\
									-v_{jr}\cdot \big(\widetilde{P}^{(1)}(\mu_j,\omega_j){\bf 1}^\top\big)_1+u_{jr}\cdot\big(\widetilde{P}^{(1)}(\mu_j,\omega_j){\bf 1}^\top\big)_2
								\end{array}\right)\nonumber\\
								&=\begin{bmatrix}
									\big(\widetilde{P}^{(1)}(\mu_j,\omega_j){\bf 1}^\top\big)_1&\big(\widetilde{P}^{(1)}(\mu_j,\omega_j){\bf 1}^\top\big)_2\\
									\big(\widetilde{P}^{(1)}(\mu_j,\omega_j){\bf 1}^\top\big)_2&-\big(\widetilde{P}^{(1)}(\mu_j,\omega_j){\bf 1}^\top\big)_1
								\end{bmatrix}\cdot
								\left(\begin{array}{cc}
									u_{jr}\\
									v_{jr}
								\end{array}\right).\label{prop-2-71}
							\end{align}
						\end{small}
						Let $$C_j=\begin{bmatrix}
							\big(\widetilde{P}^{(1)}(\mu_j,\omega_j){\bf 1}^\top\big)_1&\big(\widetilde{P}^{(1)}(\mu_j,\omega_j){\bf 1}^\top\big)_2\\
							\big(\widetilde{P}^{(1)}(\mu_j,\omega_j){\bf 1}^\top\big)_2&-\big(\widetilde{P}^{(1)}(\mu_j,\omega_j){\bf 1}^\top\big)_1
						\end{bmatrix}.$$
						
						Similarly, since $\mu_j\pm \imath \omega_{j}$ are $n_j$-multiple poles of rational function $-\alpha A(sI-A)^{-1}{\bf 1}^\top$, and from(\ref{inverse-J-2}),  $\Big((s-\mu_j)^2+\omega_j^2\Big)^{-n_j}$ only appears in the first two columns {of $\Big(sI- {\cal J}_{n_j}(\mu_j,\omega_j)\Big)^{-1}$}. We must have  $\Big(\big(\widetilde{P}^{(1)}(\mu_j,\omega_j){\bf 1}^\top\big)_1\Big)^2+\Big(\big(\widetilde{P}^{(1)}(\mu_j,\omega_j){\bf 1}^\top\big)_2\Big)^2 \neq 0$. This gives that $C_j$ is nonsingular. Then from \eqref{equ_complex}-\eqref{prop-2-71}, we can uniquely construct $F_{n_j}(\mu_j,\omega_j)$ such that \eqref{equ:cec} holds by the following way:
						\begin{align*}
							&\left(\begin{array}{cc}
								u_{j1}\\
								v_{j1}
							\end{array}\right)=C^{-1}_j{\bf 1}^\top;\\
							&\left(\begin{array}{cc}
								u_{jr}\\
								v_{jr}
							\end{array}\right)=C^{-1}_j\left[{\bf 1}^\top-\sum_{l=1}^{r-1}D_{jl}\widetilde{P}^{(r+1-l)}(\mu_j,\omega_j){\bf 1}^\top \right], \quad 2\leq r\leq n_j.
						\end{align*}
						Note that $(u_{j1})^2+(v_{j1})^2\neq 0$. Then $F_{n_j}(\mu_j,\omega_j)$ is nonsingular and (\ref{equ:cec}) holds.
						Thus the proposition is proved. 	
					\end{proof}
					
					\section{Proof of Lemma \ref{obj-convexity}}
					
					\begin{proof}
						These two convexities are straightforward. Now prove \eqref{conv_pr}.
						Note that
						\begin{align*}
							P(\lambda_i)A- {\cal J}_{n_i}(\lambda_i)P(\lambda_i)
							&=\left[\begin{array}{cccc}
								P_1(\lambda_i) \\
								P_2(\lambda_i) \\
								\vdots\\
								P_{n_i}(\lambda_i)
							\end{array}
							\right]\cdot A-\left[\begin{array}{cccc}
								\lambda_iP_1(\lambda_i) \\
								\lambda_iP_2(\lambda_i) +P_1(\lambda_i) \\
								\vdots\\
								\lambda_iP_{n_i}(\lambda_i) +P_{n_i-1}(\lambda_i)
							\end{array}
							\right]\\
							& =\left[\begin{array}{cccc}
								P_1(\lambda_i)\cdot(A-\lambda_iI)\\
								P_2(\lambda_i)\cdot(A-\lambda_iI)-P_1(\lambda_i)\\
								\vdots\\
								P_{n_i}(\lambda_i)\cdot(A-\lambda_iI)-P_{n_i-1}(\lambda_i)
							\end{array}
							\right]
						\end{align*}
						and
						\begin{align*}
							\overline{P}(\lambda_i)B_{\lambda_i}&=\Big(P_1(\lambda_i),\ldots,P_{n_i}(\lambda_i)\Big)\cdot\begin{bmatrix}
								A-\lambda_iI&-I & &\\
								& A-\lambda_iI &\ddots &\\
								& & \ddots & -I\\
								& & &A-\lambda_iI
							\end{bmatrix}_{n_i\times n_i}\\
							&=\Big(P_1(\lambda_i)\cdot(A-\lambda_iI),\ldots, P_{n_i}(\lambda_i)\cdot(A-\lambda_iI)-P_{n_i-1}(\lambda_i)\Big).
						\end{align*}
						Then
						\begin{align*}
							&\Big\|P(\lambda_i)A- {\cal J}_{n_i}(\lambda_i)P(\lambda_i) \Big\|^2_F\\
							&=P_1(\lambda_i)\cdot(A-\lambda_iI)\cdot(A^\top-\lambda_iI)\cdot P_1^\top(\lambda_i)\\
							& \quad +\sum_{r=2}^{n_i}\Big(P_r(\lambda_i)\cdot(A-\lambda_iI)-P_{r-1}(\lambda_i)\Big)\cdot\Big(P_r(\lambda_i)
							\cdot(A-\lambda_iI)-P_{r-1}(\lambda_i)\Big)^\top\\
							&=\overline{P}(\lambda_i)B_{\lambda_i}B^\top_{\lambda_i}\overline{P}^\top(\lambda_i).
						\end{align*}
						This gives \eqref{conv_pr}. For the second part, similarly,
						\begin{align*}
							&P(\mu_{j},\omega_j)A- {\cal J}_{n_j}(\mu_{j},\omega_j)P(\mu_{j},\omega_j)\\
							&=\left[\begin{array}{cccc}
								P^{(1)}_1(\mu_{j},\omega_j)\\
								P^{(1)}_2(\mu_{j},\omega_j)\\
								P^{(2)}_1(\mu_{j},\omega_j)\\
								P^{(2)}_2(\mu_{j},\omega_j)\\
								\vdots\\
								P^{(n_j)}_1(\mu_{j},\omega_j)\\
								P^{(n_j)}_2(\mu_{j},\omega_j)
							\end{array}
							\right]\cdot A
							-\left[\begin{array}{cccc}
								\mu_jP^{(1)}_1(\mu_{j},\omega_j)-\omega_jP^{(1)}_2(\mu_{j},\omega_j)\\
								\mu_jP^{(1)}_2(\mu_{j},\omega_j)+\omega_jP^{(1)}_1(\mu_{j},\omega_j)\\
								\mu_jP^{(2)}_1(\mu_{j},\omega_j)-\omega_jP^{(2)}_2(\mu_{j},\omega_j)+P^{(1)}_1(\mu_{j},\omega_j)\\
								\mu_jP^{(2)}_2(\mu_{j},\omega_j)+\omega_jP^{(2)}_1(\mu_{j},\omega_j)+P^{(1)}_2(\mu_{j},\omega_j)\\
								\vdots\\
								\mu_jP^{(n_j)}_1(\mu_{j},\omega_j)-\omega_jP^{(n_j)}_2(\mu_{j},\omega_j)+P^{(n_j-1)}_1(\mu_{j},\omega_j)\\
								\mu_jP^{(n_j)}_2(\mu_{j},\omega_j)+\omega_jP^{(n_j)}_1(\mu_{j},\omega_j)+P^{(n_j-1)}_2(\mu_{j},\omega_j)
							\end{array}
							\right]\\
							&=\left[\begin{array}{cccc}
								P^{(1)}_1(\mu_{j},\omega_j)\cdot(A-\mu_jI)+\omega_jP^{(1)}_2(\mu_{j},\omega_j)\\
								P^{(1)}_2(\mu_{j},\omega_j)\cdot(A-\mu_jI)-\omega_jP^{(1)}_1(\mu_{j},\omega_j)\\
								P^{(2)}_1(\mu_{j},\omega_j)\cdot(A-\mu_jI)+\omega_jP^{(2)}_2(\mu_{j},\omega_j)-P^{(1)}_1(\mu_{j},\omega_j)\\
								P^{(2)}_2(\mu_{j},\omega_j)\cdot(A-\mu_jI)-\omega_jP^{(2)}_1(\mu_{j},\omega_j)-P^{(1)}_2(\mu_{j},\omega_j)\\
								\vdots\\
								P^{(n_j)}_1(\mu_{j},\omega_j)\cdot(A-\mu_jI)+\omega_jP^{(n_j)}_2(\mu_{j},\omega_j)-P^{(n_j-1)}_1(\mu_{j},\omega_j)\\
								P^{(n_j)}_2(\mu_{j},\omega_j)\cdot(A-\mu_jI)-\omega_jP^{(n_j)}_1(\mu_{j},\omega_j)-P^{(n_j-1)}_2(\mu_{j},\omega_j)
							\end{array}
							\right]
						\end{align*}
						and
						\begin{align*}
							\overline{P}(\mu_{j},\omega_j)B_{(\mu_{j},\omega_j)}&=\Big(P^{(1)}_1(\mu_{j},\omega_j),P^{(1)}_2(\mu_{j},\omega_j),\ldots,  P^{(n_j)}_1(\mu_{j},\omega_j),P^{(n_j)}_2(\mu_{j},\omega_j)\Big)\\
							& \quad \times
							\begin{bmatrix}
								A_{(\mu_j,\omega_j)}&-I & &\\
								& A_{(\mu_j,\omega_j)} &\ddots &\\
								& & \ddots & -I\\
								& & &A_{(\mu_j,\omega_j)}
							\end{bmatrix}.
						\end{align*}
						Then \begin{align*}
							&\Big\|P(\mu_{j},\omega_j)A- {\cal J}_{n_j}(\mu_{j},\omega_j)P(\mu_{j},\omega_j) \Big\|^2_F\\
							&=\Big(P^{(1)}_1(\mu_{j},\omega_j) \cdot(A-\mu_jI)+\omega_jP^{(1)}_2(\mu_{j},\omega_j) \Big)\\
							&\quad\quad\quad\times
							\Big(P^{(1)}_1(\mu_{j},\omega_j) \cdot(A-\mu_jI)+\omega_jP^{(1)}_2(\mu_{j},\omega_j) \Big)^\top\\
							&\quad +\Big(P^{(1)}_2(\mu_{j},\omega_j) \cdot(A-\mu_jI)-\omega_jP^{(1)}_1(\mu_{j},\omega_j) \Big)\\
							&\quad\quad\quad\times
							\Big(P^{(1)}_2(\mu_{j},\omega_j) \cdot(A-\mu_jI)-\omega_jP^{(1)}_1(\mu_{j},\omega_j) \Big)^\top\\
							&\quad +\sum_{r=2}^{n_j}\Big(P^{(r)}_1(\mu_{j},\omega_j) \cdot(A-\mu_jI)+\omega_jP^{(r)}_2(\mu_{j},\omega_j)_{(2)}-P^{(r-1)}_1(\mu_{j},\omega_j) \Big)\\
							&\quad \quad \times\Big(P^{(r)}_1(\mu_{j},\omega_j) \cdot(A-\mu_jI)+\omega_jP^{(r)}_2(\mu_{j},\omega_j) -P^{(r-1)}_1(\mu_{j},\omega_j) \Big)^\top\\
							&\quad +\sum_{r=2}^{n_j}\Big(P^{(r)}_2(\mu_{j},\omega_j) \cdot(A-\mu_jI)-\omega_jP^{(r)}_1(\mu_{j},\omega_j) -P^{(r-1)}_2(\mu_{j},\omega_j) \Big)\\
							&\quad \quad \times\Big(P^{(r)}_2(\mu_{j},\omega_j) \cdot(A-\mu_jI)-\omega_jP^{(r)}_1(\mu_{j},\omega_j) -P^{(r-1)}_2(\mu_{j},\omega_j) \Big)^\top\\
							&=\overline{P}(\mu_{j},\omega_j)B_{(\mu_{j},\omega_j)}B_{(\mu_{j},\omega_j)}^\top\overline{P}(\mu_{j},\omega_j)^\top.
						\end{align*}	
						Hence, we have the lemma.
					\end{proof}
					
					\
					
					\section{Proof of Lemma \ref{Alter-optimal-solution}}
					
					
					\begin{proof}
						We first prove the first part. By \eqref{equ_ob}-\eqref{conv_pc},
						\begin{align}\label{obb}
							\begin{aligned}
								\Big\|PA-{\cal J}P \Big\|^2_F=&\Big(\overline{P}(\lambda_1),\ldots,\overline{P}(\lambda_{\ell}),\overline{P}(\mu_{\ell+1},\omega_{\ell+1}),\ldots,\overline{P}(\mu_{k},\omega_k)\Big) B\cdot B^\top\\ &\quad \times \Big(\overline{P}(\lambda_1),\ldots,\overline{P}(\lambda_{\ell}),\overline{P}(\mu_{\ell+1},\omega_{\ell+1}),\ldots,\overline{P}(\mu_{k},\omega_k)\Big)^{\top}
							\end{aligned}
						\end{align}
						with
						$
						B={\sf diag}\Big(
						B_{\lambda_1}, \cdots,  B_{\lambda_{\ell}},  B_{(\mu_{\ell+1},\omega_{\ell+1})},\cdots,
						B_{(\mu_{k},\omega_k)}\Big).
						$
						Note that $B\cdot B^\top$ is positive semi-definite. Then we use Corollary 2.1 of \cite{lee2005quadratic} to show that (OP[$A$]) given by (\ref{july-15-4}) has an optimal solution. To transform it into the problem addressed in Corollary 2.1 of \cite{lee2005quadratic}, take $D$, $c$, $b$ and $A$ in Problem (2.1) on p.29 studied by Corollary 2.1 of \cite{lee2005quadratic} to be
						\begin{eqnarray*}
							D=B\cdot B^\top; \ c={\bf 0}; \  b=\left(
							\begin{array}{ccc}
								{\bf 1}^\top\\
								-{\bf 1}^\top\\
								{\bf 0}^\top
							\end{array}
							\right); \ A=\begin{bmatrix}
								E\\
								-E\\
								I[\beta]
							\end{bmatrix}
						\end{eqnarray*}
						with
						\[
						E=\begin{bmatrix}
							{\bf 1} & & &\\
							&{\bf 1}& &\\
							&  & \ddots &\\
							& &  &{\bf 1}
						\end{bmatrix}\in {\cal R}^{n\times n^2}_+
						\ \mbox{and} \ \ I[\beta]=\begin{bmatrix}\beta_1I,\cdots,\beta_nI
						\end{bmatrix}\in {\cal R}^{n\times n^2}.
						\]
						Since the constraint set of {\bf (OP[$A$])} is nonempty by assumption, from  Corollary 2.1 of \cite{lee2005quadratic}, it suffices to prove that if $V \in {\cal R}^{n\times n}$ satisfies $V {\bf 1}^\top ={\bf 0}^\top$, $\beta V\geq {\bf 0}$ and
						\begin{align}\label{1212-20}
							\begin{aligned}
								&\Big(\overline{V}(\lambda_1),\ldots,\overline{V}(\lambda_{\ell}),\overline{V}(\mu_{\ell+1},\omega_{\ell+1}),\ldots,\overline{V}(\mu_{k},\omega_k)\Big) B\cdot B^\top\\ &\quad \times \Big(\overline{V}(\lambda_1),\ldots,\overline{V}(\lambda_{\ell}),\overline{V}(\mu_{\ell+1},\omega_{\ell+1}),\ldots,\overline{V}(\mu_{k},\omega_k)\Big)^{\top}=0,
							\end{aligned}
						\end{align}
						then for $P$ satisfying $P {\bf 1}^\top ={\bf 1}^\top$ and $\beta P\geq {\bf 0}$, we have
						\begin{align}\label{1212-21}
							\begin{aligned}
								&\Big(\overline{P}(\lambda_1),\ldots,\overline{P}(\lambda_{\ell}),\overline{P}(\mu_{\ell+1},\omega_{\ell+1}),\ldots,\overline{P}(\mu_{k},\omega_k)\Big) B\cdot B^\top\\ &\quad \times \Big(\overline{V}(\lambda_1),\ldots,\overline{V}(\lambda_{\ell}),\overline{V}(\mu_{\ell+1},\omega_{\ell+1}),\ldots,\overline{V}(\mu_{k},\omega_k)\Big)^{\top}\geq0.
							\end{aligned}
						\end{align}
						Here we make ${\cal J}$-blocked  partition for $V$, and similar to (\ref{pbar1})-(\ref{pbar2}),
						\begin{align*}
							&\overline{V}(\lambda_i)=\Big(V_1(\lambda_i),V_2(\lambda_i),\ldots,V_{n_i}(\lambda_i)\Big);\\	&\overline{V}(\mu_{j},\omega_j)=\Big(V^{(1)}_1(\mu_{j},\omega_j),V^{(1)}_2(\mu_{j},\omega_j),\ldots,V^{(n_j)}_1(\mu_{j},\omega_j),
							V^{(n_j)}_2(\mu_{j},\omega_j)\Big).
						\end{align*}
						It follows from \eqref{1212-20} that
						\begin{align*}
							B^\top\cdot \Big(\overline{V}(\lambda_1),\ldots,\overline{V}(\lambda_{\ell}),\overline{V}(\mu_{\ell+1},
							\omega_{\ell+1}),\ldots,\overline{V}(\mu_{k},\omega_k)\Big)^{\top}={\bf 0}.
						\end{align*}
						Then (\ref{1212-21}) holds, and the first part of the lemma is proved.
						
						For the second part, the conclusion is immediate because the constraint set  is compact and the objective function,
						$\|PA-{\cal J}P \|^2_F$, is continuous with respect to $n^2$ variables given by the entries of $A$ for any fixed $P$.
					\end{proof}

				\end{document}